\documentclass[preprint,10pt]{elsarticle}

\usepackage[T1]{fontenc}
\usepackage[english]{babel}

\usepackage{amsmath,amssymb,amsthm,mathtools}
\usepackage{graphicx}
\usepackage{accents}
\usepackage{hhline}
\usepackage{lineno}

\usepackage[colorlinks=true,
linkcolor=blue,
citecolor=blue,
urlcolor=blue]{hyperref}

\numberwithin{equation}{section}
\modulolinenumbers[5]

\newtheorem{thm}{Theorem}[section]
\newtheorem{lem}[thm]{Lemma}
\newtheorem{cor}[thm]{Corollary}

\newtheorem{rem}[thm]{Remark}
\newtheorem{defi}[thm]{Definition}

\begin{document}
	
	\begin{frontmatter}
		
		\title{On the Minimax Bifurcation Formula}
		
		\author[inst1]{Y. Sh. Il'yasov\corref{cor1}}
		\ead{ilyasov02@gmail.com}
		
		\address[inst1]{Institute of Mathematics, Ufa Federal Research Centre, RAS,
			Chernyshevsky str. 112, 450008 Ufa, Russia}
		
		\cortext[cor1]{Corresponding author}
		
	\begin{abstract}
		We introduce a Rayleigh-quotient minimax method for locating maximal
		one-sided saddle-node bifurcations in nonlinear equations, including
		non-variational ones. The method avoids branch continuation and instead
		selects the critical point directly through the minimax bifurcation formula
		\[
		\lambda^*
		:=
		\sup_{u\in\mathcal U^o}
		\inf_{v\in\mathcal S\setminus\{0\}}
		\mathcal R(u,v),
		\]
		where \(\mathcal R\) is a two-variable extended Rayleigh quotient on fixed
		cones. A saddle point of this quotient simultaneously determines the
		critical parameter, the bifurcation solution, and the adjoint singularity
		relation. This gives a direct characterization of the bifurcation threshold
		and leads to Galerkin minimax approximations, a posteriori parameter-margin
		estimates, and perturbation bounds for the critical value. The abstract
		assumptions are verified for nonlinear elliptic systems, including
		non-potential systems.
	\end{abstract}
	
		
	\end{frontmatter}

\section{Introduction}
\label{sec.Intr}

Saddle-node bifurcations, also known as folds or turning points, are among the
basic mechanisms by which solution branches of nonlinear equations appear or
disappear. In nonlinear differential equations, they often mark critical
parameter values beyond which solutions in a prescribed class cease to exist.
Locating such values is therefore a central problem both in analysis and in
numerical computation.

The standard way to detect saddle-node bifurcations is to continue a branch of
solutions and monitor the loss of invertibility of the linearized operator.
This is the basis of classical local and global bifurcation theory, including
the Crandall--Rabinowitz theorem, Rabinowitz's global theorem, and the
bifurcation theory of Krasno\-sel'ski\u{\i}, as well as of numerical
continuation methods
\cite{CranRibin,keller1,kielh,Kuznet,Seid,Rabin1971,KrasnosBif}. These
methods are powerful, but they are essentially indirect: the critical point is
detected through the behaviour of a solution curve. This becomes especially
delicate in infinite-dimensional problems, in systems, and in models where
positivity or cone constraints are an essential part of the problem.

This paper develops a direct approach to this problem. Maximal one-sided
saddle-node points are identified by means of a \emph{minimax bifurcation
	formula}. Rather than constructing or following a solution branch, the
critical parameter is characterized directly as an extremal value of an
extended Rayleigh quotient, so that the minimax formula itself determines the
saddle-node value.

Let \(\mathcal U^o,\mathcal S\subset W\) be prescribed cones in a Banach
space \(W\). We consider nonlinear equations of the form
\begin{equation}
	\label{Gf-intro}
	\mathcal F(u,\lambda)
	:=
	\mathcal A(u)-\lambda\mathcal G(u)=0,\quad u\in W,
\end{equation}
where \(\mathcal A\) and \(\mathcal G\) take values in \(W^*\) and are
continuously Fr\'echet differentiable in a suitable neighbourhood of the
admissible set \(\mathcal U^o\).

The idea behind the method is simple. If \(u\) is a solution, then testing the
equation against \(v\in\mathcal S\) gives the extended equation
\cite{IlyasFunc}
\[
\langle\mathcal A(u),v\rangle
=
\lambda \langle\mathcal G(u),v\rangle .
\]
Thus, under the denominator positivity condition imposed below, the parameter
\(\lambda\) can be recovered from the extended Rayleigh quotient
\[
\mathcal R(u,v)
:=
\frac{\langle\mathcal A(u),v\rangle}
{\langle\mathcal G(u),v\rangle},\quad \langle\mathcal G(u),v\rangle \neq 0.
\]
We call \(\mathcal R\) the extended Rayleigh quotient
\cite{IlyasJDE24}. Here \(\langle\cdot,\cdot\rangle\) denotes the duality
pairing between \(W^*\) and \(W\).

For an admissible solution \(u\), this quotient is independent of \(v\) and
equals the corresponding parameter \(\lambda\). This leads to the \emph{minimax
bifurcation formula}
\begin{equation}
	\label{MainB}
	\lambda^*
	:=
	\sup_{u\in\mathcal U^o}
	\inf_{v\in\mathcal S\setminus\{0\}}
	\mathcal R(u,v).
\end{equation}
The number \(\lambda^*\) is the natural variational candidate for the maximal
parameter value at which admissible solutions may exist. Indeed, every
admissible solution satisfies
\[
\mathcal R(u,v)=\lambda
\qquad
\forall v\in\mathcal S\setminus\{0\},
\]
and therefore \(\lambda\le\lambda^*\). Hence, if \(\lambda^*<+\infty\),
equation \eqref{Gf-intro} has no admissible solutions for
\(\lambda>\lambda^*\), so that \(\lambda^*\) provides a variational upper
threshold for solvability.

The role of the second variable is crucial. The extended quotient does not
only provide scalar tests for the parameter; it also carries the dual
information needed to identify the point at which the equation becomes
singular. Indeed, suppose, formally, that a pair \((u^*,v^*)\) realizes the
minimax value and satisfies the corresponding Euler conditions
\[
D_v\mathcal R(u^*,v^*)=0,
\qquad
D_u\mathcal R(u^*,v^*)=0 .
\]
Then the first condition recovers the equation itself,
\[
\mathcal F(u^*,\lambda^*)=0,
\]
whereas the second gives the adjoint kernel relation
\[
\big\langle
D_u\mathcal F(u^*,\lambda^*)\varphi,
v^*
\big\rangle=0
\qquad
\forall \varphi\in W .
\]
Thus the minimax construction is designed not merely to detect a critical
parameter value, but to select a singular triple
\((u^*,v^*,\lambda^*)\), and hence a singular point
\((u^*,\lambda^*)\). In this sense, the quotient is genuinely extended: the
additional variable is not an auxiliary testing parameter, but the dual
variable through which the singularity condition is encoded.

This variational viewpoint also gives a new way to quantify the distance to
the threshold. For a known admissible solution \((u_\lambda,\lambda)\), the
gap \(\lambda^*-\lambda\) can be estimated directly through the extended
Rayleigh quotient \(\mathcal R\), without computing the principal eigenvalue
of the linearized operator \(D_u\mathcal F(u_\lambda,\lambda)\). This
provides computable a posteriori certificates in threshold problems such as
pull-in instability, voltage collapse, and tipping phenomena, and also yields
perturbation bounds for the maximal one-sided saddle-node value.

The paper develops this program in an abstract cone setting. We prove that the
saddle-point realization of the minimax formula selects a maximal one-sided
saddle-node point, establish the stability of this construction under
Galerkin approximation, derive a posteriori parameter-margin certificates and
perturbation bounds, and verify the abstract assumptions in model elliptic
problems.

The approach builds on earlier work on generalized Rayleigh quotients and
inverse variational characterizations
\cite{IlyasovCRAS,IlyasFunc,IvanIlya,IlIvan1,IlyasChaos,IlyasJDE24,
	IlyasovPOMI25,Salazar,ilyasELA}. The novelty of the present paper is the
two-cone minimax formulation, its abstract saddle-point interpretation, its
stability under Galerkin approximation, and its use in computable
certificates for maximal one-sided saddle-node values.

The paper is organized as follows. Section~\ref{sec.1} develops the abstract
minimax theory. Section~\ref{sec6} records a posteriori and perturbation
estimates for the maximal one-sided saddle-node value. Section~\ref{sec.13}
contains the proofs of the main abstract results. Sections~\ref{sec.4} and
\ref{sec.5} apply the theory to model elliptic problems.
Section~\ref{sec:ConcRem} contains concluding remarks. The appendices collect
the finite-dimensional cone setting, interpolation estimates, and
Picone-type inequalities used in the compactness arguments.

Throughout the paper, \(c,C,c_1,C_1,\ldots\) denote positive constants whose
values may change from line to line.

\section{Minimax bifurcation formula for an abstract problem}
\label{sec.1}

Throughout this section, \(W\) denotes a reflexive Banach space and
\(\mathcal C^0\) an ordered Banach function space. The duality pairing between
\(W^*\) and \(W\) is denoted by \(\langle\cdot,\cdot\rangle\), and all order
relations are understood in \(\mathcal C^0\). We work with a convex cone
\(\mathcal S^o\subset W\cap\mathcal C^0\) and put
\[
\mathcal S:=\overline{\mathcal S^o}^{\,\mathcal C^0}\cap W .
\]
The cone is assumed to be generating in the sense that
$
\overline{\operatorname{span}\mathcal S^o}^{\,W}=W .
$
\par \noindent
Finally, we fix a subclass 
\[\mathcal U^o\subset\mathcal S^o,\]
 which will
serve as the admissible class of regular solutions.

Finite-dimensional approximations are described by subspaces
\(W_r\subset W\), \(N_r:=\dim W_r\), such that
\(\bigcup_{r\ge1}W_r\) is dense in \(W\), together with interpolation
operators \(\mathcal I_r:\mathcal C^0\to W_r\). We assume the following
compatibility condition for the Galerkin cones.
\begin{description}
	\item[{\rm (C)}]
	The Galerkin cones satisfy the following conditions.
	\begin{description}
		\item[{\rm (c1)}]
		For every \(r\ge1\), \(\mathcal S_r^o\subset W_r\) is a nonempty open
		cone, \(\mathcal S_r:=\overline{\mathcal S_r^o}^{\,W_r}\), and
		\(\mathcal S_r\subset\mathcal S\).
		
		\item[{\rm (c2)}]
		The cone \(\mathcal S_r\) is polyhedral: there is a basis
		\(\{\eta_1,\ldots,\eta_{N_r}\}\) of \(W_r\) such that
		\[
		\mathcal S_r^o
		=
		\left\{
		v=\sum_{i=1}^{N_r}v^i\eta_i\in W_r:
		v^i>0,\ i=1,\ldots,N_r
		\right\}.
		\]
		
		\item[{\rm (c3)}]
		For every \(u\in\mathcal U^o\), the interpolant \(\mathcal I_ru\) is
		well defined and, for all sufficiently large \(r\),	\(\mathcal I_ru\in\mathcal S_r^o\) .
	\end{description}
\end{description}

The central object of the abstract framework is the following one-parameter
equation on the cone \(\mathcal S\):
\begin{equation}\label{Gf}
	\mathcal F(u,\lambda):=\mathcal A(u)-\lambda\mathcal G(u)=0,
	\qquad u\in\mathcal S .
\end{equation}
Here \(\lambda\in\mathbb R\) is the bifurcation parameter, and
\(\mathcal A,\mathcal G:W\to W^*\) are continuously Fr\'echet differentiable
on \(\mathcal U^o\). Thus, for \(u\in\mathcal U^o\), the linearization is well
defined and is given by
\[
D_u\mathcal F(u,\lambda)\xi
=
D_u\mathcal A(u)\xi-\lambda D_u\mathcal G(u)\xi,
\qquad \xi\in W .
\]

We impose the following structural assumptions.
\begin{description}
	\item[{\rm (R)}]
	If \(u\in\mathcal S^o\) weakly satisfies \(\mathcal F(u,\lambda)=0\), then
	\(u\in\mathcal U^o\).
	
	\item[{\rm (D)}]
	\(\langle\mathcal G(u),v\rangle>0\) for every
	\((u,v)\in\mathcal S^o\times(\mathcal S\setminus\{0\})\).
\end{description}
We shall regard \({\rm (R)}\) and \({\rm (D)}\) as standing assumptions.

\begin{defi}
	A solution \((u_0,\lambda_0)\in\mathcal U^o\times\mathbb R\) of
	\eqref{Gf} is called \emph{singular} if there exists
	\(v_0\in\mathcal S\setminus\{0\}\) such that
	\[
	\bigl\langle D_u\mathcal F(u_0,\lambda_0)\xi,v_0\bigr\rangle=0
	\qquad \forall \xi\in W .
	\]
	It is called a \emph{one-sided saddle-node point} in \(\mathcal U^o\) if
	there exist \(\varepsilon>0\) and a neighbourhood \(V\) of \(u_0\) in \(W\)
	such that \eqref{Gf} has no solutions in \(\mathcal U^o\cap V\) for
	\(\lambda\in(\lambda_0,\lambda_0+\varepsilon)\). It is called
	\emph{maximal} if every one-sided saddle-node point
	\((\hat u,\hat\lambda)\in\mathcal U^o\times\mathbb R\) satisfies
	\(\hat\lambda\le\lambda_0\).
\end{defi}

The basic tool of the method is the extended Rayleigh quotient associated with
\eqref{Gf}. It is defined by
\[
\mathcal R(u,v):=
\frac{\langle\mathcal A(u),v\rangle}
{\langle\mathcal G(u),v\rangle},
\qquad
\langle\mathcal G(u),v\rangle\neq0 .
\]
Under \({\rm (D)}\), this quotient is well defined on
\(\mathcal S^o\times(\mathcal S\setminus\{0\})\). For each fixed
\(u\in\mathcal S^o\), the map \(v\mapsto\mathcal R(u,v)\) is continuous on
\(\mathcal S\setminus\{0\}\) in both the strong and weak topologies of \(W\).
Moreover, \(\mathcal R\) is continuously Fr\'echet differentiable on
\(\mathcal U^o\times\mathcal S^o\).

The distinguished variational quantity of the construction is the minimax level
\begin{equation}\label{eq:lambda-star}
	\lambda^*
	:=
	\sup_{u\in\mathcal U^o}
	\inf_{v\in\mathcal S\setminus\{0\}}
	\mathcal R(u,v).
\end{equation}
It will be shown below that, under the appropriate compactness and consistency
assumptions, this value is realized by a singular solution and coincides with
the maximal one-sided saddle-node threshold.

\begin{defi}
	The \emph{minimax bifurcation formula} is said to hold for \eqref{Gf}
	with respect to the cone \(\mathcal S\) if there exists
	\((u^*,v^*)\in\mathcal U^o\times(\mathcal S\setminus\{0\})\) such that
	\begin{equation}\label{MBFDefi}
		\lambda^*
		=
		\mathcal R(u^*,v^*)
		=
		\inf_{v\in\mathcal S\setminus\{0\}}\mathcal R(u^*,v)
		=
		\sup_{u\in\mathcal U^o}
		\inf_{v\in\mathcal S\setminus\{0\}}
		\mathcal R(u,v).
	\end{equation}
\end{defi}

A pair \((u_0,v_0)\in\mathcal U^o\times(\mathcal S\setminus\{0\})\) is called
a \emph{saddle-point solution pair} at the level \(\lambda_0\) if
\[
\mathcal F(u_0,\lambda_0)=0
\quad\text{and}\quad
\bigl\langle D_u\mathcal F(u_0,\lambda_0)\xi,v_0\bigr\rangle=0
\qquad \forall \xi\in W .
\]
Thus, when such a pair realizes the minimax level \(\lambda^*<+\infty\), it
also provides the singularity relation associated with the corresponding
one-sided saddle-node point.

We next pass to Galerkin approximations. The Galerkin approximation of
\eqref{Gf} is to find
\((u,\lambda)\in\mathcal S_r^o\times\mathbb R\) such that
\begin{equation}\label{eq:AppEq}
	\langle\mathcal F(u,\lambda),\zeta\rangle=0
	\qquad \forall \zeta\in W_r .
\end{equation}
For \(u\in\mathcal S_r^o\), set
\[
\lambda_r(u):=
\inf_{v\in\mathcal S_r\setminus\{0\}}\mathcal R(u,v),
\qquad
\lambda_r^*:=
\sup_{u\in\mathcal S_r^o}\lambda_r(u).
\]

We use the following quotient-compatibility condition on the Galerkin spaces.
\begin{description}
	\item[{\rm (CQ)}]
	For every \(r\ge1\),
	\[
	\langle\mathcal G(u),v\rangle>0
	\qquad
	\forall (u,v)\in
	\mathcal S_r^o\times(\mathcal S_r\setminus\{0\}).
	\]
	Moreover, the finite-dimensional restrictions needed below are \(C^1\) on
	\(\mathcal S_r^o\); in particular, \(\mathcal R\) is \(C^1\) on
	\(\mathcal S_r^o\times\mathcal S_r^o\).
\end{description}

By the continuity of \(\mathcal F\), the residual is sequentially closed in
the following sense: if \(u_n\to u\) in \(W\), \(\lambda_n\to\lambda\), and
\(\zeta_n\to\zeta\) in \(W\), with \(u_n\in\mathcal S_{r_n}^o\) and
\(\zeta_n\in W_{r_n}\), then
\[
\langle\mathcal F(u_n,\lambda_n),\zeta_n\rangle
\to
\langle\mathcal F(u,\lambda),\zeta\rangle .
\]
We impose the following closedness assumption for the linearization.
\begin{description}
	\item[{\rm (CD)}]
	For every \(u\in\mathcal U^o\), \(v\in\mathcal S\),
	\(\xi\in W\), and \(\lambda\in\mathbb R\), the following holds:
	whenever \(u_n\in\mathcal S_{r_n}^o\),
	\(v_n\in\mathcal S_{r_n}^o\), \(\xi_n\in W_{r_n}\),
	\(r_n\to\infty\), and \(\lambda_n\to\lambda\), with
	\[
	u_n\to u \quad\text{in }W,\qquad
	\xi_n\to\xi \quad\text{in }W,\qquad
	v_n\rightharpoonup v \quad\text{weakly in }W,
	\]
	one has
	\[
	\bigl\langle
	D_u\mathcal F(u_n,\lambda_n)\xi_n,v_n
	\bigr\rangle
	\to
	\bigl\langle
	D_u\mathcal F(u,\lambda)\xi,v
	\bigr\rangle .
	\]
\end{description}

Condition \({\rm (CD)}\) is automatic if \(\mathcal F\) is continuously
Fréchet differentiable with respect to \(u\) on
\(\mathcal U^o\times\mathbb R\), with values in \(\mathcal L(W,W^*)\).
Thus \({\rm (CD)}\) only has to be checked separately in applications where
the derivative is singular, or where differentiability is available only in
a relative sense along admissible Galerkin sequences.

We shall call a Galerkin scheme
\[
\{W_r,\mathcal S_r^o,\mathcal I_r\}_{r\ge1}
\]
\emph{admissible} if it satisfies
\({\rm (C)}\), \({\rm (CQ)}\), and \({\rm (CD)}\).

\subsection{Compactness and minimax theorem}

We now add the assumptions specific to the nonlinear bifurcation problem.

Let
\(\{W_r,\mathcal S_r^o,\mathcal I_r\}_{r\ge1}\) be a fixed admissible Galerkin
scheme. We assume the following compactness and boundary-exclusion condition.

\begin{description}
	\item[{\rm (H)}]
	For all sufficiently large \(r\ge1\), the following assertions hold.
	\begin{description}
		\item[{\rm (h1)}]
		If \(\lambda_r^*<+\infty\), then there exists
		\(\alpha_r<\lambda_r^*\) such that
		\[
		K_{r,\alpha_r}:=
		\{u\in\mathcal S_r^o:\lambda_r(u)\ge\alpha_r\}
		\]
		has compact closure in \(W_r\), and this closure is contained in
		\(\mathcal S_r^o\).
		
		\item[{\rm (h2)}]
		If \(u\in\mathcal S_r^o\), \(v\in\mathcal S_r\setminus\{0\}\),
		\[
		\mathcal R(u,v)=\lambda_r^*<+\infty,
		\qquad
		D_u\mathcal R(u,v)(\xi)=0\quad\forall\xi\in W_r,
		\]
		then \(v\in\mathcal S_r^o\).
	\end{description}
\end{description}

Although condition \({\rm (H)}\) is formulated in terms of a chosen
finite-dimensional approximation, its meaning is simple. Assumption
\({\rm (h1)}\) says that the relevant superlevel sets remain compactly inside
the positive Galerkin cone. Assumption \({\rm (h2)}\) excludes boundary
minimizing directions; in applications this is precisely the positivity
property expected from the maximum principle for the corresponding linearized
Galerkin problem. Thus, for the standard finite-element Galerkin cones used
below, these requirements become transparent and are verified by compactness,
comparison, and maximum-principle arguments.

\medskip
\medskip
\noindent\textbf{Extended Palais--Smale condition.}
We shall use the notation \((PS)_e\), where the subscript \(e\) stands for
``extended''. This indicates that the compactness requirement is not the
classical Palais--Smale condition for a one-variable functional, but its
two-variable Galerkin analogue adapted to the extended Rayleigh quotient.

\begin{defi}
	The Rayleigh quotient \(\mathcal R\) is said to satisfy the extended
	\((PS)_e\)-condition at level \(\lambda\in\mathbb R\) if every sequence
	\(r_n\to\infty\) and every sequence
	\((u_n,v_n)\in\mathcal S_{r_n}^o\times\mathcal S_{r_n}^o\) satisfying
	\begin{equation}\label{PSEcond}
		\begin{cases}
			\mathcal R(u_n,v_n)\to\lambda,\\
			D_u\mathcal R(u_n,v_n)(\xi)=0
			\quad \forall \xi\in W_{r_n},\\
			D_v\mathcal R(u_n,v_n)(\zeta)=0
			\quad \forall \zeta\in W_{r_n},
		\end{cases}
	\end{equation}
	has a subsequence and positive numbers \(t_n>0\) such that, after replacing
	\(v_n\) by \(t_nv_n\),
	\[
	u_n\to\bar u \quad\text{strongly in }W,
	\qquad
	v_n\rightharpoonup\bar v \quad\text{weakly in }W,
	\]
	with \(\bar u\in\mathcal S^o\) and
	\(\bar v\in\mathcal S\setminus\{0\}\). Such a sequence is called an
	extended \((PS)_e\)-sequence at level \(\lambda\).
\end{defi}

\medskip
\noindent\textbf{Singular Galerkin limits.}
The first result shows how the compactness encoded in the extended
\((PS)_e\)-condition turns finite-dimensional minimax pairs into genuine
singular solutions of the continuous problem.

\begin{thm}[Singular limit theorem]
	\label{Thm1}
	Assume that \(\{W_r,\mathcal S_r^o,\mathcal I_r\}_{r\ge1}\) is an
	admissible Galerkin scheme. Under the standing assumptions, suppose that
	condition \({\rm (H)}\) holds and that
	\(0\le\lambda_r^*<+\infty\) for all sufficiently large \(r\). Then the
	following assertions hold.
	
	\begin{description}
		\item[\((1^\circ)\)]
		For every sufficiently large \(r\), there exist
		\(u_r^*,v_r^*\in\mathcal S_r^o\) such that
		\begin{equation}\label{eq:MMth2}
			\lambda_r^*
			=
			\mathcal R(u_r^*,v_r^*)
			=
			\min_{v\in\mathcal S_r^o}\mathcal R(u_r^*,v)
			=
			\max_{u\in\mathcal S_r^o}
			\min_{v\in\mathcal S_r^o}\mathcal R(u,v).
		\end{equation}
		Moreover,
		\[
		\langle\mathcal F(u_r^*,\lambda_r^*),\zeta\rangle=0
		\quad \forall \zeta\in W_r,
		\qquad
		\bigl\langle
		D_u\mathcal F(u_r^*,\lambda_r^*)\xi,v_r^*
		\bigr\rangle=0
		\quad \forall \xi\in W_r .
		\]
		Thus \((u_r^*,v_r^*)\) is a finite-dimensional saddle-point pair at
		the level \(\lambda_r^*\).
		
		\item[\((2^\circ)\)]
		Assume, in addition, that, along a subsequence,
		\(\lambda_r^*\to\hat\lambda^*\in\mathbb R\), and that
		\(\mathcal R\) satisfies the extended \((PS)_e\)-condition at the level
		\(\hat\lambda^*\). Then, up to a further subsequence and a positive
		rescaling of \(v_r^*\),
		\[
		u_r^*\to u^* \quad\text{strongly in }W,
		\qquad
		v_r^*\rightharpoonup v^* \quad\text{weakly in }W,
		\]
		where \(u^*\in\mathcal U^o\) and
		\(v^*\in\mathcal S\setminus\{0\}\). Moreover,
		\[
		\mathcal F(u^*,\hat\lambda^*)=0
		\quad\text{in }W^*,
		\qquad
		\bigl\langle
		D_u\mathcal F(u^*,\hat\lambda^*)\xi,v^*
		\bigr\rangle=0
		\quad \forall \xi\in W .
		\]
	\end{description}
\end{thm}
\medskip

We next relate the Galerkin saddle levels obtained in Theorem~\ref{Thm1} to
the continuous minimax value. For \(u\in\mathcal U^o\), put
\[
\lambda(u):=
\inf_{v\in\mathcal S\setminus\{0\}}\mathcal R(u,v).
\]
Then \eqref{eq:lambda-star} can be written as
\[
\lambda^*=\sup_{u\in\mathcal U^o}\lambda(u).
\]
To pass from the Galerkin minimax levels to the continuous minimax value, we
need the following one-sided consistency property. It is required only near
the top level \(\lambda^*\) and says that the interpolation procedure does
not decrease the inner minimax value in the limit.

\begin{description}
	\item[{\rm (LC)}]
	Assume that \(\lambda^*<+\infty\). There exists \(\delta_0>0\) such that,
	for every \(u\in\mathcal U^o\) with
	\(\lambda(u)\ge\lambda^*-\delta_0\),
	\[
	\liminf_{r\to\infty}\lambda_r(\mathcal I_r u)
	\ge
	\lambda(u).
	\]
\end{description}

\medskip
\noindent\textbf{Main abstract theorem.}
We can now state the main result of this section. It combines the singular
Galerkin limits of Theorem~\ref{Thm1} with the one-sided consistency condition
to obtain the full minimax bifurcation formula with respect to the cone
\(\mathcal S\), and identifies the resulting value as the maximal one-sided
saddle-node threshold.

\begin{thm}[Minimax bifurcation formula]
	\label{Thm2}
	Assume the standing assumptions and \({\rm (LC)}\). Suppose that condition
	\({\rm (H)}\) holds, that the minimax levels satisfy
	\(0<c\le\lambda_r^*\le C<+\infty\) for all sufficiently large \(r\), and
	that \(\mathcal R\) satisfies the extended \((PS)_e\)-condition at every
	level \(\lambda\in[c,C]\).
	
	Then the minimax bifurcation formula holds for \eqref{Gf}. More precisely,
	there exists
	\((u^*,v^*)\in\mathcal U^o\times(\mathcal S\setminus\{0\})\) such that
	\begin{equation}\label{MBFThm}
		c\le \lambda^*
		=
		\mathcal R(u^*,v^*)
		=
		\inf_{v\in\mathcal S\setminus\{0\}}\mathcal R(u^*,v)
		=
		\sup_{u\in\mathcal U^o}
		\inf_{v\in\mathcal S\setminus\{0\}}
		\mathcal R(u,v).
	\end{equation}
	Moreover,
	\[
	\mathcal F(u^*,\lambda^*)=0
	\quad\text{in }W^*,
	\qquad
	\bigl\langle
	D_u\mathcal F(u^*,\lambda^*)\xi,v^*
	\bigr\rangle=0
	\quad \forall \xi\in W .
	\]
	In addition, \eqref{Gf} has no solution \(u\in\mathcal U^o\) for any
	\(\lambda>\lambda^*\). Consequently, \((u^*,\lambda^*)\) is a maximal
	one-sided saddle-node point in \(\mathcal U^o\).
	
	Finally, the approximating minimax levels converge to the same value:
	\(\lambda_r^*\to\lambda^*\). There exist a subsequence, still denoted by
	\(r\), and pairs
	\((u_r^*,v_r^*)\in\mathcal S_r^o\times\mathcal S_r^o\) such that, after a
	positive rescaling of \(v_r^*\),
	\[
	u_r^*\to u^* \quad\text{strongly in }W,
	\qquad
	v_r^*\rightharpoonup v^* \quad\text{weakly in }W .
	\]
\end{thm}

\begin{rem}[On the verification of the abstract assumptions]
	The compactness and consistency assumptions used above are not intended to
	be merely formal hypotheses. Their role is to isolate, in the abstract
	setting, the analytic mechanisms needed for the minimax construction to
	select a genuine singular solution.
	
	In the applications below, these assumptions are verified for concrete
	elliptic problems and admissible finite-dimensional cones. In particular,
	compactness follows from elliptic estimates and growth conditions, boundary
	exclusion is obtained from comparison and discrete maximum principles,
	the extended \((PS)_e\)-condition is obtained using Picone-type
	inequalities, and the closedness condition \({\rm (CD)}\) is checked
	separately when the concave term has a singular derivative.
\end{rem}

\begin{rem}[Dual minimax formula]
	The orientation in \eqref{eq:lambda-star} is adapted to maximal one-sided
	saddle-node values, where admissible solutions cease to exist for larger
	parameters. The opposite orientation is described by
	\[
	\lambda_*
	:=
	\inf_{u\in\mathcal U^o}
	\sup_{v\in\mathcal S\setminus\{0\}}
	\mathcal R(u,v).
	\]
	Under the corresponding reversed inequalities and compactness assumptions,
	the same argument yields minimal one-sided saddle-node values, where
	admissible solutions cease to exist below the threshold. We omit the
	details, since they are completely parallel.
\end{rem}


\section{A posteriori certificates and perturbation bounds}
\label{sec6}

The minimax bifurcation formula also gives quantitative information about the
maximal one-sided saddle-node value. This section records two simple
consequences. First, the extended Rayleigh quotient provides computable
a posteriori intervals for \(\lambda^*\), without computing the principal
eigenvalue of the linearized operator
\[
D_u\mathcal F(u_\lambda,\lambda).
\]
Second, the same quotient gives one-sided bounds on the change of \(\lambda^*\)
under perturbations of the operator.

\subsection{A posteriori bounds from the minimax inequality}

Set
\[
\lambda(u):=
\inf_{v\in\mathcal S\setminus\{0\}}\mathcal R(u,v),
\qquad
u\in\mathcal U^o,
\]
and
\[
\Lambda(v):=
\sup_{u\in\mathcal U^o}\mathcal R(u,v),
\qquad
v\in\mathcal S\setminus\{0\}.
\]
Then
\[
\lambda^*
=
\sup_{u\in\mathcal U^o}\lambda(u),
\qquad
\Lambda^*
:=
\inf_{v\in\mathcal S\setminus\{0\}}\Lambda(v),
\]
whenever \(\Lambda^*\) is well defined, and the minimax inequality gives
\[
\lambda^*\le \Lambda^* .
\]

\begin{cor}[A posteriori parameter-margin certificate]
	\label{cor:parameter-margin}
	Assume that \(\mathcal R\) is well defined on
	\(\mathcal U^o\times(\mathcal S\setminus\{0\})\). Then, for every
	\(u\in\mathcal U^o\) and every \(v\in\mathcal S\setminus\{0\}\) with
	\(\Lambda(v)<+\infty\), one has
	\begin{equation}\label{eq:weak-gap-bound}
		\lambda(u)\le \lambda^*\le \Lambda(v).
	\end{equation}
	Equivalently,
	\[
	0\le \lambda^*-\lambda(u)\le \Lambda(v)-\lambda(u).
	\]
	
	In particular, if \(u_\lambda\in\mathcal U^o\) is an admissible solution of
	\(\mathcal F(u_\lambda,\lambda)=0\), then, for every
	\(v\in\mathcal S\setminus\{0\}\) with \(\Lambda(v)<+\infty\),
	\begin{equation}\label{eq:parameter-margin-certificate}
		0\le \lambda^*-\lambda\le \Lambda(v)-\lambda .
	\end{equation}
	Consequently, if \(\Lambda(v)-\lambda\le\varepsilon\) for some
	\(v\in\mathcal S\setminus\{0\}\), then
	\[
	0\le \lambda^*-\lambda\le\varepsilon .
	\]
\end{cor}

\begin{proof}
	The lower bound in \eqref{eq:weak-gap-bound} follows directly from the
	definition of \(\lambda^*\). To prove the upper bound, fix
	\(v\in\mathcal S\setminus\{0\}\). For every \(w\in\mathcal U^o\),
	\[
	\lambda(w)
	=
	\inf_{\zeta\in\mathcal S\setminus\{0\}}\mathcal R(w,\zeta)
	\le
	\mathcal R(w,v)
	\le
	\sup_{z\in\mathcal U^o}\mathcal R(z,v)
	=
	\Lambda(v).
	\]
	Taking the supremum over \(w\in\mathcal U^o\), we obtain
	\(\lambda^*\le\Lambda(v)\).
	
	If \(u_\lambda\) solves \(\mathcal F(u_\lambda,\lambda)=0\), then
	\[
	\langle \mathcal A(u_\lambda),v\rangle
	=
	\lambda\langle\mathcal G(u_\lambda),v\rangle
	\qquad
	\forall v\in\mathcal S\setminus\{0\}.
	\]
	Hence \(\mathcal R(u_\lambda,v)=\lambda\) for all
	\(v\in\mathcal S\setminus\{0\}\), and so
	\(\lambda(u_\lambda)=\lambda\). Substituting
	\(u=u_\lambda\) into \eqref{eq:weak-gap-bound} gives
	\eqref{eq:parameter-margin-certificate}.
\end{proof}

Thus each pair \((u,v)\in\mathcal U^o\times(\mathcal S\setminus\{0\})\) with
\(\Lambda(v)<+\infty\) gives the certified localization
\[
\lambda^*\in[\lambda(u),\Lambda(v)].
\]
The width of this interval is
\[
{\rm gap}(u,v):=\Lambda(v)-\lambda(u).
\]
If the minimax equality \(\lambda^*=\Lambda^*\) is not known, this width may
also contain the duality gap \(\Lambda^*-\lambda^*\). Nevertheless, a small
value of \({\rm gap}(u,v)\) still gives a computable a posteriori certificate
for the critical parameter.

\begin{rem}
	Estimate \eqref{eq:parameter-margin-certificate} is particularly useful in
	threshold problems, such as pull-in instability, voltage collapse, and
	tipping phenomena; see, for instance,
	\cite{Ashwin2012,DobsonChiang,Dobson1992,PeleskoBernstein,Salazar,Scheffer2009}.
	In such problems one wants to estimate how close a given admissible
	solution \((u_\lambda,\lambda)\) is to the loss of solvability.
\end{rem}

\subsection{Perturbation estimate}

For an operator \(\mathcal A\), set
\[
\mathcal F_{\mathcal A}(u,\lambda)
:=
\mathcal A(u)-\lambda\mathcal G(u),
\qquad
\mathcal R_{\mathcal A}(u,v)
:=
\frac{\langle\mathcal A(u),v\rangle}
{\langle\mathcal G(u),v\rangle}.
\]

\begin{cor}[One-sided perturbation bound]
	\label{Th3}
	Let \(\mathcal P:\mathcal U^o\to W^*\), and assume that
	\[
	-\infty<
	\lambda^*_{\mathcal A+\mathcal P}
	:=
	\sup_{u\in\mathcal U^o}
	\inf_{v\in\mathcal S\setminus\{0\}}
	\mathcal R_{\mathcal A+\mathcal P}(u,v)
	<+\infty .
	\]
	If \(u^*_{\mathcal A}\in\mathcal U^o\) satisfies
	\[
	\lambda^*_{\mathcal A}
	=
	\inf_{v\in\mathcal S\setminus\{0\}}
	\mathcal R_{\mathcal A}(u^*_{\mathcal A},v),
	\]
	then
	\begin{equation}\label{eq:PertEST}
		\lambda^*_{\mathcal A+\mathcal P}
		\ge
		\lambda^*_{\mathcal A}
		+
		\inf_{v\in\mathcal S\setminus\{0\}}
		\frac{\langle \mathcal P(u^*_{\mathcal A}),v\rangle}
		{\langle\mathcal G(u^*_{\mathcal A}),v\rangle}.
	\end{equation}
\end{cor}

\begin{proof}
	Since
	\[
	\mathcal R_{\mathcal A+\mathcal P}(u,v)
	=
	\mathcal R_{\mathcal A}(u,v)
	+
	\frac{\langle \mathcal P(u),v\rangle}
	{\langle\mathcal G(u),v\rangle},
	\]
	we test the outer supremum in the definition of
	\(\lambda^*_{\mathcal A+\mathcal P}\) at
	\(u=u^*_{\mathcal A}\). Using
	\(\inf(a+b)\ge\inf a+\inf b\), we obtain
	\[
	\lambda^*_{\mathcal A+\mathcal P}
	\ge
	\inf_{v\in\mathcal S\setminus\{0\}}
	\mathcal R_{\mathcal A+\mathcal P}(u^*_{\mathcal A},v)
	\ge
	\lambda^*_{\mathcal A}
	+
	\inf_{v\in\mathcal S\setminus\{0\}}
	\frac{\langle \mathcal P(u^*_{\mathcal A}),v\rangle}
	{\langle\mathcal G(u^*_{\mathcal A}),v\rangle}.
	\]
	This proves \eqref{eq:PertEST}.
\end{proof}


\section{Proof of Theorems~\ref{Thm1}, \ref{Thm2}}
\label{sec.13}

\begin{proof}[Proof of Theorem~\ref{Thm1}]
	We first record two elementary identities used below. Since
	\[
	\mathcal R(u,v)
	=
	\frac{\langle\mathcal A(u),v\rangle}
	{\langle\mathcal G(u),v\rangle},
	\qquad
	\mathcal F(u,\lambda)=\mathcal A(u)-\lambda\mathcal G(u),
	\]
	the quotient rule gives, whenever \(\langle\mathcal G(u),v\rangle\ne0\),
	\[
	D_u\mathcal R(u,v)(\xi)
	=
	\frac{
		\bigl\langle
		D_u\mathcal F(u,\mathcal R(u,v))\xi,v
		\bigr\rangle
	}{
		\langle\mathcal G(u),v\rangle
	},
	\qquad
	D_v\mathcal R(u,v)(\zeta)
	=
	\frac{
		\bigl\langle
		\mathcal F(u,\mathcal R(u,v)),\zeta
		\bigr\rangle
	}{
		\langle\mathcal G(u),v\rangle
	}.
	\]
	Thus, at a point where \(\mathcal R(u,v)=\lambda\), the condition
	\(D_u\mathcal R(u,v)=0\) is equivalent to
	\[
	\bigl\langle D_u\mathcal F(u,\lambda)\xi,v\bigr\rangle=0
	\qquad \forall \xi,
	\]
	whereas \(D_v\mathcal R(u,v)=0\) is equivalent to
	\[
	\bigl\langle \mathcal F(u,\lambda),\zeta\bigr\rangle=0
	\qquad \forall \zeta .
	\]
	
	We first prove \((1^\circ)\). Fix \(r\) sufficiently large and set
	\[
	\lambda_r(u):=
	\inf_{v\in\mathcal S_r\setminus\{0\}}\mathcal R(u,v),
	\qquad u\in\mathcal S_r^o .
	\]
	Let \(\{\eta_1,\ldots,\eta_{N_r}\}\) be the basis associated with
	\(\mathcal S_r\) in \({\rm (c2)}\). Then, by \({\rm (CQ)}\),
	\(b_i(u):=\langle\mathcal G(u),\eta_i\rangle>0\). Hence, for
	\(v=\sum_i v^i\eta_i\in\mathcal S_r\setminus\{0\}\), \(v^i\ge0\), we have
	\[
	\mathcal R(u,v)
	=
	\sum_{i=1}^{N_r}
	\frac{v^i b_i(u)}{\sum_j v^j b_j(u)}
	\mathcal R(u,\eta_i),
	\]
	and therefore
	\[
	\lambda_r(u)=\min_{1\le i\le N_r}\mathcal R(u,\eta_i).
	\]
	By the \(C^1\)-part of \({\rm (CQ)}\), the functions
	\(u\mapsto\mathcal R(u,\eta_i)\) are \(C^1\) on \(\mathcal S_r^o\).
	Hence \(\lambda_r\) is continuous on \(\mathcal S_r^o\).
	
	By \({\rm (h1)}\), \(\lambda_r\) attains its maximum. Indeed, any
	maximizing sequence eventually belongs to \(K_{r,\alpha_r}\), whose
	closure is compact and contained in \(\mathcal S_r^o\). Hence there exists
	\(u_r^*\in\mathcal S_r^o\) such that
	\[
	\lambda_r^*
	=
	\lambda_r(u_r^*)
	=
	\inf_{v\in\mathcal S_r\setminus\{0\}}\mathcal R(u_r^*,v).
	\]
	
	The point \((u_r^*,\lambda_r^*)\) solves the finite-dimensional constrained
	problem
	\[
	\text{maximize }\lambda
	\quad\text{over }(u,\lambda)\in\mathcal S_r^o\times\mathbb R,
	\qquad
	\lambda-\mathcal R(u,\eta_i)\le0,\quad i=1,\ldots,N_r .
	\]
	Since \(u_r^*\) is an interior point of \(\mathcal S_r^o\), no boundary
	multiplier in the \(u\)-variable appears. By the \(C^1\)-part of
	\({\rm (CQ)}\) and the Fritz John multiplier rule, there exist
	\(\mu_0\ge0\) and \(\mu_i\ge0\), not all zero, such that
	\[
	\mu_0=\sum_{i=1}^{N_r}\mu_i,\qquad
	\sum_{i=1}^{N_r}\mu_iD_u\mathcal R(u_r^*,\eta_i)=0,
	\qquad
	\mu_i\bigl(\lambda_r^*-\mathcal R(u_r^*,\eta_i)\bigr)=0 .
	\]
	Then \(\sum_i\mu_i>0\). After normalization, assume
	\(\sum_i\mu_i=1\), and set
	\[
	b_i:=\langle\mathcal G(u_r^*),\eta_i\rangle,
	\qquad
	v_r^*:=
	\sum_{i=1}^{N_r}\frac{\mu_i}{b_i}\eta_i .
	\]
	Then \(v_r^*\in\mathcal S_r\setminus\{0\}\). By the complementarity
	relations,
	\[
	\langle\mathcal A(u_r^*),v_r^*\rangle
	=
	\sum_i\mu_i\mathcal R(u_r^*,\eta_i)
	=
	\lambda_r^*,
	\qquad
	\langle\mathcal G(u_r^*),v_r^*\rangle=1,
	\]
	and hence
	\[
	\mathcal R(u_r^*,v_r^*)=\lambda_r^*.
	\]
	
	For every \(\xi\in W_r\), again using complementarity and the identity for
	\(D_u\mathcal R\) recorded above,
	\[
	0
	=
	\sum_i\mu_iD_u\mathcal R(u_r^*,\eta_i)(\xi)
	=
	\sum_i
	\frac{\mu_i}{b_i}
	\bigl\langle
	D_u\mathcal F(u_r^*,\lambda_r^*)\xi,\eta_i
	\bigr\rangle .
	\]
	Therefore
	\[
	\bigl\langle
	D_u\mathcal F(u_r^*,\lambda_r^*)\xi,v_r^*
	\bigr\rangle=0
	\qquad \forall \xi\in W_r .
	\]
	Since \(\mathcal R(u_r^*,v_r^*)=\lambda_r^*\), this is equivalent to
	\(D_u\mathcal R(u_r^*,v_r^*)(\xi)=0\) for all \(\xi\in W_r\). Hence, by
	\({\rm (h2)}\), \(v_r^*\in\mathcal S_r^o\).
	
	Thus \(v_r^*\) is an interior minimizer of
	\(v\mapsto\mathcal R(u_r^*,v)\). Consequently,
	\[
	D_v\mathcal R(u_r^*,v_r^*)(\zeta)=0
	\qquad \forall \zeta\in W_r .
	\]
	Since \(\mathcal R(u_r^*,v_r^*)=\lambda_r^*\), the identity for
	\(D_v\mathcal R\) gives
	\[
	\langle\mathcal F(u_r^*,\lambda_r^*),\zeta\rangle=0
	\qquad \forall \zeta\in W_r .
	\]
	
	Thus \((u_r^*,v_r^*)\) is a finite-dimensional saddle-point pair at the
	level \(\lambda_r^*\). Moreover,
	\[
	\lambda_r^*
	=
	\mathcal R(u_r^*,v_r^*)
	=
	\min_{v\in\mathcal S_r^o}\mathcal R(u_r^*,v)
	=
	\max_{u\in\mathcal S_r^o}
	\min_{v\in\mathcal S_r^o}\mathcal R(u,v).
	\]
	This proves \((1^\circ)\).
	
	We prove \((2^\circ)\). Since the Galerkin scheme is admissible,
	\({\rm (CD)}\) holds. Suppose that, along a subsequence,
	\(\lambda_r^*\to\hat\lambda^*\). By \((1^\circ)\),
	\[
	\mathcal R(u_r^*,v_r^*)=\lambda_r^*\to\hat\lambda^*,
	\qquad
	D_u\mathcal R(u_r^*,v_r^*)=0\ \text{on }W_r,
	\qquad
	D_v\mathcal R(u_r^*,v_r^*)=0\ \text{on }W_r .
	\]
	Thus \((u_r^*,v_r^*)\) is an extended \((PS)_e\)-sequence at level
	\(\hat\lambda^*\). By the extended \((PS)_e\)-condition, after a further
	subsequence and a positive rescaling of \(v_r^*\),
	\[
	u_r^*\to u^* \quad\text{strongly in }W,
	\qquad
	v_r^*\rightharpoonup v^* \quad\text{weakly in }W,
	\]
	with \(u^*\in\mathcal S^o\) and \(v^*\in\mathcal S\setminus\{0\}\).
	
	The rescaling of \(v_r^*\) preserves the identities involving
	\(D_u\mathcal F\), while the Galerkin equation is independent of \(v_r^*\).
	Let \(\zeta\in W\), and choose \(\zeta_r\in W_r\) such that
	\(\zeta_r\to\zeta\) strongly in \(W\). Since
	\[
	\langle\mathcal F(u_r^*,\lambda_r^*),\zeta_r\rangle=0,
	\]
	the sequential closedness of the residual gives
	\[
	\mathcal F(u^*,\hat\lambda^*)=0
	\quad\text{in }W^* .
	\]
	By \({\rm (R)}\), we then have \(u^*\in\mathcal U^o\).
	
	Let \(\xi\in W\), and choose \(\xi_r\in W_r\) such that
	\(\xi_r\to\xi\) strongly in \(W\). Since
	\[
	\bigl\langle
	D_u\mathcal F(u_r^*,\lambda_r^*)\xi_r,v_r^*
	\bigr\rangle=0,
	\]
	condition \({\rm (CD)}\) yields
	\[
	\bigl\langle
	D_u\mathcal F(u^*,\hat\lambda^*)\xi,v^*
	\bigr\rangle=0
	\qquad \forall \xi\in W .
	\]
	This proves \((2^\circ)\).
\end{proof}
\medskip
\begin{proof}[Proof of Theorem~\ref{Thm2}]
	Recall that
	\[
	\lambda^*
	=
	\sup_{u\in\mathcal U^o}
	\lambda(u),
	\qquad
	\lambda(u):=
	\inf_{v\in\mathcal S\setminus\{0\}}\mathcal R(u,v).
	\]
	We first prove the liminf estimate
	\begin{equation}\label{eq:liminf-Th2}
		\liminf_{r\to\infty}\lambda_r^*\ge \lambda^* .
	\end{equation}
	Let \(\varepsilon\in(0,\delta_0)\). By the definition of \(\lambda^*\),
	there exists \(u_\varepsilon\in\mathcal U^o\) such that
	\[
	\lambda(u_\varepsilon)>\lambda^*-\varepsilon .
	\]
	By \({\rm (c3)}\), for all sufficiently large \(r\),
	\(\mathcal I_ru_\varepsilon\in\mathcal S_r^o\). Hence
	\[
	\lambda_r^*
	\ge
	\lambda_r(\mathcal I_ru_\varepsilon)
	=
	\inf_{v\in\mathcal S_r\setminus\{0\}}
	\mathcal R(\mathcal I_ru_\varepsilon,v).
	\]
	Taking the lower limit and using \({\rm (LC)}\), we obtain
	\[
	\liminf_{r\to\infty}\lambda_r^*
	\ge
	\lambda(u_\varepsilon)
	>
	\lambda^*-\varepsilon .
	\]
	Letting \(\varepsilon\downarrow0\), we get \eqref{eq:liminf-Th2}.
	
	By the assumed bounds on \(\lambda_r^*\), every subsequence of
	\((\lambda_r^*)\) has a further convergent subsequence. Let
	\[
	\lambda_{r_j}^*\to\bar\lambda,
	\qquad
	\bar\lambda\in[c,C].
	\]
	Then \eqref{eq:liminf-Th2} gives \(\lambda^*\le\bar\lambda\).
	
	By Theorem~\ref{Thm1}\((1^\circ)\), for every sufficiently large \(j\)
	there exist Galerkin minimax pairs
	\[
	(u_{r_j}^*,v_{r_j}^*)
	\in
	\mathcal S_{r_j}^o\times\mathcal S_{r_j}^o .
	\]
	Since \(\mathcal R\) satisfies the extended \((PS)_e\)-condition at every
	level in \([c,C]\), Theorem~\ref{Thm1}\((2^\circ)\) yields, after passing
	to a further subsequence and rescaling the second component, a pair
	\[
	\bar u\in\mathcal U^o,
	\qquad
	\bar v\in\mathcal S\setminus\{0\},
	\]
	such that
	\[
	u_{r_j}^*\to\bar u \quad\text{strongly in }W,
	\qquad
	v_{r_j}^*\rightharpoonup\bar v \quad\text{weakly in }W,
	\]
	and
	\[
	\mathcal F(\bar u,\bar\lambda)=0
	\quad\text{in }W^*,
	\qquad
	\bigl\langle
	D_u\mathcal F(\bar u,\bar\lambda)\xi,\bar v
	\bigr\rangle=0
	\quad \forall \xi\in W .
	\]
	Since \(\bar u\in\mathcal U^o\subset\mathcal S^o\), condition
	\({\rm (D)}\) applies. Therefore
	\[
	\mathcal R(\bar u,v)=\bar\lambda
	\qquad
	\forall v\in\mathcal S\setminus\{0\},
	\]
	and hence
	\[
	\bar\lambda
	=
	\inf_{v\in\mathcal S\setminus\{0\}}\mathcal R(\bar u,v)
	\le
	\sup_{u\in\mathcal U^o}
	\inf_{v\in\mathcal S\setminus\{0\}}\mathcal R(u,v)
	=
	\lambda^* .
	\]
	Together with \(\lambda^*\le\bar\lambda\), this gives
	\(\bar\lambda=\lambda^*\).
	
	Thus every subsequence of \((\lambda_r^*)\) has a further subsequence
	converging to \(\lambda^*\). Consequently, \(\lambda_r^*\to\lambda^*\).
	In particular, \(\lambda^*\in[c,C]\), and hence \(\lambda^*>0\).
	
	Applying the preceding compactness argument to the convergent sequence
	\(\lambda_r^*\to\lambda^*\), we obtain, after passing to a subsequence and
	rescaling \(v_r^*\), a pair
	\[
	(u^*,v^*)\in\mathcal U^o\times(\mathcal S\setminus\{0\})
	\]
	such that
	\[
	u_r^*\to u^* \quad\text{strongly in }W,
	\qquad
	v_r^*\rightharpoonup v^* \quad\text{weakly in }W,
	\]
	and
	\[
	\mathcal F(u^*,\lambda^*)=0
	\quad\text{in }W^*,
	\qquad
	\bigl\langle
	D_u\mathcal F(u^*,\lambda^*)\xi,v^*
	\bigr\rangle=0
	\quad \forall \xi\in W .
	\]
	Again, since \(u^*\in\mathcal U^o\subset\mathcal S^o\), condition
	\({\rm (D)}\) gives
	\[
	\mathcal R(u^*,v)=\lambda^*
	\qquad
	\forall v\in\mathcal S\setminus\{0\}.
	\]
	In particular, \(\mathcal R(u^*,v^*)=\lambda^*\), and therefore
	\[
	\lambda^*
	=
	\mathcal R(u^*,v^*)
	=
	\inf_{v\in\mathcal S\setminus\{0\}}\mathcal R(u^*,v)
	=
	\sup_{u\in\mathcal U^o}
	\inf_{v\in\mathcal S\setminus\{0\}}
	\mathcal R(u,v).
	\]
	Thus the minimax bifurcation formula holds, and \((u^*,v^*)\) is a
	saddle-point solution pair at the level \(\lambda^*\).
	
	It remains to identify this value as the maximal one-sided threshold.
	Indeed, if \(u\in\mathcal U^o\) is a solution of \eqref{Gf} at some
	parameter \(\lambda\), then condition \({\rm (D)}\) gives
	\[
	\mathcal R(u,v)=\lambda
	\qquad
	\forall v\in\mathcal S\setminus\{0\}.
	\]
	Hence \(\lambda=\lambda(u)\le\lambda^*\). Therefore \eqref{Gf} has no
	solution in \(\mathcal U^o\) for any \(\lambda>\lambda^*\).
	
	Finally, every one-sided saddle-node point in \(\mathcal U^o\) is, by
	definition, a solution of \eqref{Gf}. Hence its parameter cannot exceed
	\(\lambda^*\). Since \((u^*,\lambda^*)\) is a singular solution and no
	admissible solutions exist above \(\lambda^*\), it is a maximal one-sided
	saddle-node point in \(\mathcal U^o\).
	
	The convergence of the Galerkin minimax levels and, along a subsequence,
	of the corresponding minimax pairs has already been proved.
\end{proof}

\section{Existence of singular solutions for elliptic systems}
\label{sec.4}

We now apply the abstract minimax principle to non-variational elliptic
systems. Let \(\Omega\subset\mathbb R^d\), \(d\ge1\), be a bounded domain,
sufficiently regular for the Hopf boundary principle to apply. We also assume
that \(\Omega\) admits an exhaustion from inside by polyhedral domains
\(\Omega_r\subset\Omega\), \(r=1,2,\ldots\), as in
Appendix~\ref{sec:appendix-FE}. The associated conforming triangulations are
chosen so that the discrete maximum principle \eqref{eq:impl} holds. The
index \(r\) includes both the choice of the subdomain and the mesh refinement.

Let \(m\ge1\). We consider the system
\begin{equation}
	\label{f}
	\left\{
	\begin{aligned}
		\mathcal L u^k
		&=\lambda (u^k)^{q-1}+f_k(x,u),
		&& x\in\Omega,\quad k=1,\ldots,m,\\
		u^k&\ge0,
		&& x\in\Omega,\quad k=1,\ldots,m,\\
		u^k&=0,
		&& x\in\partial\Omega,\quad k=1,\ldots,m,
	\end{aligned}
	\right.
\end{equation}
where \(u=(u^1,\ldots,u^m)\) and \(1<q<2\). For simplicity, the same scalar
elliptic operator acts in each component:
\begin{equation}\label{eq:mathL}
	\mathcal L w
	=
	-\frac{\partial}{\partial x_\alpha}
	\left(
	\sigma_{\alpha\beta}(x)
	\frac{\partial w}{\partial x_\beta}
	\right)
	+c(x)w ,
\end{equation}
where summation over repeated spatial indices is understood. We assume that
\(\sigma_{\alpha\beta},c\in C(\overline\Omega)\), that
\((\sigma_{\alpha\beta})\) is symmetric and uniformly elliptic, and that the
bilinear form
\begin{equation}\label{eq:awz}
	a(w,z):=
	\int_\Omega
	\left(
	\sigma_{\alpha\beta}(x)
	\frac{\partial w}{\partial x_\beta}
	\frac{\partial z}{\partial x_\alpha}
	+c(x)wz
	\right)\,dx
\end{equation}
is coercive on \(H_0^1(\Omega)\). Set
\[
\mathbb W:=(H_0^1(\Omega))^m,
\qquad
\mathbb L^p:=(L^p(\Omega))^m,
\]
and
\[
a_m(u,v):=\sum_{k=1}^m a(u^k,v^k),
\qquad
\langle u,v\rangle_m:=\sum_{k=1}^m\int_\Omega u^k v^k\,dx .
\]

We write
\[
\mathcal L u:=(\mathcal L u^1,\ldots,\mathcal L u^m)^T,
\qquad
f(u):=(f_1(\cdot,u),\ldots,f_m(\cdot,u))^T,
\]
and
\[
g(u):=\bigl((u^1)^{q-1},\ldots,(u^m)^{q-1}\bigr)^T .
\]
Then \eqref{f} has the abstract form
\[
\mathcal F(u,\lambda):=\mathcal L u-f(u)-\lambda g(u)=0.
\]

We impose the following assumptions. For each \(i=1,\ldots,m\),
\(f_i(\cdot,u)\in C(\overline\Omega)\),
\(f_i(x,\cdot)\in C^1(\mathbb R_+^m)\), and \(f_i(x,0)=0\). Moreover, there
exist \(2<\gamma_1\le\gamma_2<2^*\) and \(C>0\) such that, for
\(i,j=1,\ldots,m\),
\begin{align}
	0\le f_i(x,u)
	&\le
	C\bigl(|u|^{\gamma_1-1}+|u|^{\gamma_2-1}\bigr),
	\label{42}
	\\
	0\le f_{i,u^j}(x,u)
	&\le
	C\bigl(|u|^{\gamma_1-2}+|u|^{\gamma_2-2}\bigr),
	\label{422}
\end{align}
and, for \(i\ne j\),
\[
f_{i,u^j}(x,u)>0
\qquad
\text{for a.e. }x\in\Omega,\quad u\in(0,+\infty)^m .
\]

We also assume the following large-growth condition.

\begin{description}
	\item[\({\rm (LG)}\)]
	There exist \(R>0\) and \(\delta_0>0\) such that, denoting by
	\(\lambda_1\) the principal Dirichlet eigenvalue of \(\mathcal L\),
	\[
	\sum_{k=1}^m f_k(x,u)
	\ge
	(\lambda_1+\delta_0)\sum_{k=1}^m u^k
	\]
	for a.e. \(x\in\Omega\) and all \(u\in\mathbb R_+^m\) with \(|u|>R\).
\end{description}

Finally, we require the radial monotonicity condition
\begin{equation}
	\label{eq:radial-monotonicity-f}
	\bigl\langle f_u(u)u-f(u),v\bigr\rangle_m\ge0
	\qquad
	u\in\mathbb U^o,\quad v\in\mathbb S\setminus\{0\},
\end{equation}
and, for some \(\theta>1\), the Picone-type superlinearity condition
\begin{equation}
	\label{eq:picone-superlinearity-f}
	Q_f(x,u)\ge \theta\sum_{i=1}^m f_i(x,u)u^i
	\qquad
	\text{for a.e. }x\in\Omega,\quad u\in\mathbb R_+^m,
\end{equation}
where
\[
Q_f(x,u):=
\sum_{i=1}^m f_{i,u^i}(x,u)(u^i)^2
+
2\sum_{1\le i<j\le m}
\sqrt{f_{i,u^j}(x,u)f_{j,u^i}(x,u)}\,u^iu^j .
\]

We use the scalar cone
\[
S^o
:=
\left\{
w\in H_0^1(\Omega)\cap C(\overline\Omega):
w>0 \ \text{in }\Omega
\right\},
\qquad
S:=\overline{S^o}^{\,C(\overline\Omega)}\cap H_0^1(\Omega),
\]
and its componentwise version
\[
\mathbb S^o:=(S^o)^m,
\qquad
\mathbb S:=\overline{\mathbb S^o}^{\,\mathbb W}.
\]
The regular cone is \(\mathbb U^o:=(U^o)^m\), where
\[
U^o
:=
\left\{
w\in S^o:
\mathcal Lw\in L^\infty_{\rm loc}(\Omega),
\quad
\mathcal Lw\ge0
\text{ in }\mathcal D'(\Omega)
\right\}.
\]

By the Hopf-type lower bound, every \(w\in U^o\) satisfies
\[
w\ge c_w d_\Omega \quad\text{in }\Omega,
\qquad
d_\Omega(x):=\operatorname{dist}(x,\partial\Omega),
\]
with some \(c_w>0\). Hence
Lemma~\ref{lem:concave-differentiability} implies that
\(u\mapsto g(u)\) is differentiable on \(\mathbb U^o\), and
\(D_u\mathcal F(u,\lambda)\) is well defined for every
\(u\in\mathbb U^o\).

We now verify the abstract regularity condition \({\rm (R)}\) at
nonnegative levels, in the positive class considered here. Let
\(u\in\mathbb S^o\) be a weak solution of \eqref{f} with \(\lambda\ge0\).
Then, for each \(k=1,\ldots,m\),
\[
\mathcal L u^k
=
\lambda (u^k)^{q-1}+f_k(x,u)
\in L^\infty_{\rm loc}(\Omega),
\qquad
\mathcal L u^k\ge0
\quad\text{in }\mathcal D'(\Omega).
\]
The strong maximum principle and the Hopf boundary principle therefore give
\(u^k\in U^o\) for every \(k=1,\ldots,m\). Thus \(u\in\mathbb U^o\).

The extended Rayleigh quotient associated with \eqref{f} is
\[
\mathcal R(u,v):=
\frac{
	a_m(u,v)-\langle f(u),v\rangle_m
}{
	\langle g(u),v\rangle_m
},
\qquad
u\in\mathbb S^o,\quad v\in\mathbb S\setminus\{0\}.
\]
Its denominator is positive on this set. The corresponding
infinite-dimensional minimax level is
\[
\lambda^*
:=
\sup_{u\in\mathbb U^o}
\inf_{v\in\mathbb S\setminus\{0\}}
\mathcal R(u,v).
\]
We use the finite-dimensional Galerkin cones constructed in
Appendix~\ref{sec:appendix-FE}. Thus \(W_r\subset H_0^1(\Omega_r)\) is
extended by zero to \(\Omega\), so that \(W_r\subset H_0^1(\Omega)\), and
\(S_r^o\) is the corresponding positive nodal cone.
We use the product spaces
and cones
\[
\mathbb W_r:=(W_r)^m,
\qquad
\mathbb S_r^o:=(S_r^o)^m,
\qquad
\mathbb S_r:=(S_r)^m .
\]
The componentwise interpolation operator is denoted again by
\(\mathcal I_r\). The finite-dimensional minimax levels are
\[
\lambda_r^*
:=
\sup_{u\in\mathbb S_r^o}
\inf_{v\in\mathbb S_r\setminus\{0\}}
\mathcal R(u,v).
\]

\begin{lem}[Uniform bounds for the minimax levels]
	\label{lem:bounds-lambda}
	Assume the hypotheses stated above, including \({\rm (LG)}\), and suppose
	that the Galerkin cones and the discrete maximum-principle structure are those
	described in Appendix~\ref{sec:appendix-FE}.
	Then
	\[
	\lambda^*<+\infty .
	\]
	Moreover, there exist constants \(c,C>0\) and \(r_0\ge1\), independent of
	\(r\), such that
	\[
	c\le \lambda_r^*\le C,
	\qquad r\ge r_0 .
	\]
\end{lem}

\begin{proof}
	Let \((\lambda_1,\phi_1)\) be the principal Dirichlet eigenpair of
	\(\mathcal L\), with \(\phi_1>0\) in \(\Omega\). Since \(\mathcal L\) is
	symmetric,
	\[
	a(w,\phi_1)=\lambda_1\int_\Omega w\phi_1\,dx
	\qquad \forall w\in H_0^1(\Omega).
	\]
	Set \(\boldsymbol\phi_1:=(\phi_1,\ldots,\phi_1)\). Testing the inner
	infimum by \(\boldsymbol\phi_1\), we get
	\[
	\lambda^*
	\le
	\sup_{u\in\mathbb U^o}
	\frac{
		\displaystyle
		\int_\Omega
		\left(
		\lambda_1\sum_{k=1}^m u^k
		-
		\sum_{k=1}^m f_k(x,u)
		\right)\phi_1\,dx
	}{
		\displaystyle
		\int_\Omega
		\sum_{k=1}^m (u^k)^{q-1}\phi_1\,dx
	}.
	\]
	For \(\mu\ge0\), define
	\[
	H_\mu(x,u):=
	\frac{
		\mu\sum_{k=1}^m u^k-\sum_{k=1}^m f_k(x,u)
	}{
		\sum_{k=1}^m (u^k)^{q-1}
	},
	\qquad
	u\in\mathbb R_+^m\setminus\{0\}.
	\]
	By the growth assumptions on \(f\), \(1<q<2\), and \({\rm (LG)}\),
	\[
	\sup_{x\in\Omega,\ u\in\mathbb R_+^m\setminus\{0\}}
	H_{\lambda_1}(x,u)<+\infty .
	\]
	Hence \(\lambda^*<+\infty\).
	
	The discrete upper bound is identical. Let
	\((\lambda_{1,r},\phi_{1,r})\) be the discrete principal eigenpair. By
	admissibility,
	\[
	\lambda_{1,r}\to\lambda_1,
	\qquad
	\phi_{1,r}\in S_r^o .
	\]
	Thus, for all sufficiently large \(r\),
	\(\lambda_{1,r}\le\lambda_1+\delta_0/2\). The same pointwise bound for
	\(H_\mu\), uniformly for
	\(\mu\in[0,\lambda_1+\delta_0/2]\), and the test
	\(\boldsymbol\phi_{1,r}:=(\phi_{1,r},\ldots,\phi_{1,r})\in\mathbb S_r^o\)
	give
	\[
	\lambda_r^*\le C .
	\]
	
	It remains to prove the positive lower bound. Let \(\omega_r\in W_r\) be
	the discrete torsion function,
	\[
	a(\omega_r,\psi_i)=\int_\Omega\psi_i\,dx,
	\qquad i=1,\ldots,N_r .
	\]
	By the discrete maximum principle and the admissibility assumptions,
	\[
	\omega_r\in S_r^o,
	\qquad
	\|\omega_r\|_{L^\infty(\Omega)}\le C ,
	\]
	with \(C\) independent of \(r\). Set
	\(\boldsymbol\omega_r:=(\omega_r,\ldots,\omega_r)\). For
	\(0<t\le1\) and \(v\in\mathbb S_r\setminus\{0\}\), the definition of
	\(\omega_r\) gives
	\[
	\frac{a_m(t\boldsymbol\omega_r,v)}
	{\langle g(t\boldsymbol\omega_r),v\rangle_m}
	\ge c_1t^{2-q}.
	\]
	On the other hand, using
	\(1<q<2<\gamma_1\le\gamma_2\) and
	\(\|\omega_r\|_{L^\infty}\le C\),
	\[
	\frac{\langle f(t\boldsymbol\omega_r),v\rangle_m}
	{\langle g(t\boldsymbol\omega_r),v\rangle_m}
	\le
	c_2t^{\gamma_1-q}+c_3t^{\gamma_2-q}.
	\]
	Therefore
	\[
	\mathcal R(t\boldsymbol\omega_r,v)
	\ge
	c_1t^{2-q}
	-
	c_2t^{\gamma_1-q}
	-
	c_3t^{\gamma_2-q}.
	\]
	Since \(\gamma_1>2\), we fix \(t_0\in(0,1]\), independent of \(r\), such
	that the right-hand side is bounded below by some \(c>0\). Hence
	\[
	\lambda_r^*
	\ge
	\inf_{v\in\mathbb S_r\setminus\{0\}}
	\mathcal R(t_0\boldsymbol\omega_r,v)
	\ge c
	\]
	for all sufficiently large \(r\). The proof is complete.
\end{proof}

\begin{lem}\label{lem:PS-application}
	Under the assumptions stated above, suppose in addition that the discrete
	comparison estimate of Lemma~\ref{lem:discrete-comparison-sublinear}
	holds for the Galerkin operator \(A_r\). Then \(\mathcal R\) satisfies the
	extended \((PS)_e\)-condition at every level \(\lambda>0\).
\end{lem}

\begin{proof}
	Let \(r_n\to\infty\), and let
	\((u_n,v_n)\in\mathbb S_{r_n}^o\times\mathbb S_{r_n}^o\) be an extended
	\((PS)_e\)-sequence at a level \(\lambda>0\). Set
	\[
	\lambda_n:=\mathcal R(u_n,v_n)\to\lambda .
	\]
	By the homogeneity of \(\mathcal R\) with respect to the second variable,
	we normalize \(a_m(v_n,v_n)=1\). The Galerkin criticality relations,
	equivalent to \(D_u\mathcal R(u_n,v_n)=D_v\mathcal R(u_n,v_n)=0\) on
	\(\mathbb W_{r_n}\), are
	\[
	\langle\mathcal F(u_n,\lambda_n),\zeta\rangle=0
	\quad \forall \zeta\in\mathbb W_{r_n},
	\qquad
	\bigl\langle D_u\mathcal F(u_n,\lambda_n)\xi,v_n\bigr\rangle=0
	\quad \forall \xi\in\mathbb W_{r_n}.
	\]
	Testing the first identity by \(u_n\), we get
	\begin{equation}\label{eq:energy-ps-g}
		a_m(u_n,u_n)
		=
		\lambda_n\langle g(u_n),u_n\rangle_m
		+
		\langle f(u_n),u_n\rangle_m .
	\end{equation}
	
	We first prove that \((u_n)\) is bounded in \(\mathbb W\). Let
	\(\theta_n\in\mathbb W_{r_n}\) be the nodal function defined componentwise by
	\[
	(\theta_n)_i^k=\frac{(u_{n,i}^k)^2}{v_{n,i}^k}.
	\]
	Since \(u_n,v_n\in\mathbb S_{r_n}^o\), this function is well defined. Taking
	\(\xi=\theta_n\) in the adjoint Galerkin identity, using the componentwise
	discrete Picone inequality, and applying the arithmetic--geometric mean
	inequality to the cooperative off-diagonal terms, we obtain
	\[
	a_m(u_n,u_n)
	\ge
	(q-1)\lambda_n\langle g(u_n),u_n\rangle_m
	+
	\int_\Omega Q_f(x,u_n)\,dx .
	\]
	By \eqref{eq:picone-superlinearity-f}, for some \(\eta>1\),
	\[
	a_m(u_n,u_n)
	\ge
	(q-1)\lambda_n\langle g(u_n),u_n\rangle_m
	+
	\eta\langle f(u_n),u_n\rangle_m .
	\]
	Combining this estimate with \eqref{eq:energy-ps-g}, we find
	\[
	(\eta-1)a_m(u_n,u_n)
	\le
	(\eta-q+1)\lambda_n\langle g(u_n),u_n\rangle_m .
	\]
	Since
	\[
	\langle g(u_n),u_n\rangle_m
	=
	\|u_n\|_{\mathbb L^q}^q
	\le
	Ca_m(u_n,u_n)^{q/2},
	\]
	and \(q<2\), the sequence \((u_n)\) is bounded in \(\mathbb W\).
	
	By the normalization of \(v_n\) and the coercivity of \(a_m\), the sequence
	\((v_n)\) is also bounded in \(\mathbb W\). Hence, after passing to a
	subsequence,
	\[
	u_n\rightharpoonup u^*,
	\qquad
	v_n\rightharpoonup v^*
	\quad\text{weakly in }\mathbb W,
	\]
	and
	\[
	u_n\to u^*,
	\qquad
	v_n\to v^*
	\quad\text{strongly in }\mathbb L^\beta,\quad 1\le\beta<2^*,
	\]
	and a.e. in \(\Omega\). Since \(\mathbb S\) is weakly closed and
	\(\mathbb S_{r_n}\subset\mathbb S\), we have \(u^*,v^*\in\mathbb S\).
	
	We now pass to the limit in the Galerkin equation. Let
	\(\zeta\in\mathbb W\). Choose \(\zeta_n\in\mathbb W_{r_n}\) such that
	\[
	\zeta_n\to\zeta
	\qquad\text{strongly in }\mathbb W .
	\]
	Then
	\[
	\langle\mathcal F(u_n,\lambda_n),\zeta_n\rangle=0.
	\]
	Using the weak convergence of \(u_n\), the strong convergence in
	subcritical Lebesgue spaces, and the growth assumptions on \(f\), we obtain
	\[
	\langle\mathcal F(u^*,\lambda),\zeta\rangle=0
	\qquad \forall \zeta\in\mathbb W .
	\]
	
	We next show that \(u^*\in\mathbb S^o\). Since
	\(\lambda_n\to\lambda>0\), for all large \(n\),
	\[
	\mathcal L u_n^k
	=
	\lambda_n(u_n^k)^{q-1}+f_k(x,u_n)
	\ge
	\frac{\lambda}{2}(u_n^k)^{q-1}
	\]
	in the discrete weak sense, for \(k=1,\ldots,m\). By
	Lemma~\ref{lem:discrete-comparison-sublinear}, applied componentwise,
	\[
	u_n^k\ge w_n^k,
	\]
	where \(w_n^k\in S_{r_n}^o\) is the Galerkin solution of the scalar
	sublinear problem
	\[
	\mathcal L w_n^k=\frac{\lambda}{2}(w_n^k)^{q-1}.
	\]
	The standard convergence of Galerkin solutions for this scalar sublinear
	problem gives
	\[
	w_n^k\to w
	\quad\text{strongly in }H_0^1(\Omega)
	\quad\text{and a.e. in }\Omega,
	\]
	where \(w\) is the positive weak solution of
	\[
	\mathcal L w=\frac{\lambda}{2}w^{q-1},
	\qquad
	w\in H_0^1(\Omega),
	\qquad
	w>0 \quad\text{in }\Omega .
	\]
	Passing to the limit in \(u_n^k\ge w_n^k\), we obtain
	\[
	u^{*,k}\ge w>0
	\quad\text{a.e. in }\Omega,
	\qquad k=1,\ldots,m .
	\]
	Since \(u^*\) solves the system and the right-hand side has subcritical
	growth, standard bootstrap/Moser estimates and boundary H\"older regularity
	imply \(u^*\in(C(\overline\Omega))^m\). Hence
	\(u^{*,k}>0\) in \(\Omega\), \(k=1,\ldots,m\), and therefore
	\(u^*\in\mathbb S^o\).
	
	It remains to strengthen the convergence of \(u_n\) and to prove that
	\(v^*\ne0\). Testing the limit equation by \(u^*\), we get
	\[
	a_m(u^*,u^*)
	=
	\lambda\langle g(u^*),u^*\rangle_m
	+
	\langle f(u^*),u^*\rangle_m .
	\]
	Comparing this identity with \eqref{eq:energy-ps-g}, and using the compact
	convergence of the nonlinear terms, we obtain
	\[
	a_m(u_n,u_n)\to a_m(u^*,u^*) .
	\]
	Together with the weak convergence \(u_n\rightharpoonup u^*\), the
	coercivity of \(a_m\) gives
	\[
	u_n\to u^*
	\qquad\text{strongly in }\mathbb W .
	\]
	
	Finally, taking \(\xi=v_n\) in the adjoint Galerkin identity and using the
	normalization \(a_m(v_n,v_n)=1\), we obtain
	\[
	1
	=
	(q-1)\lambda_n
	\left\langle u_n^{q-2}v_n,v_n\right\rangle_m
	+
	\left\langle f_u(u_n)v_n,v_n\right\rangle_m .
	\]
	On the other hand, testing the Galerkin equation by the nodal function with
	values
	\[
	\frac{(v_{n,i}^k)^2}{u_{n,i}^k}
	\]
	and using the componentwise discrete Picone inequality yields
	\[
	1
	\ge
	\lambda_n
	\left\langle u_n^{q-2}v_n,v_n\right\rangle_m
	+
	\left\langle f(u_n),\frac{v_n^2}{u_n}\right\rangle_m .
	\]
	Choose \(\chi\in(q-1,1)\). Multiplying the last inequality by \(\chi\) and
	subtracting it from the preceding identity gives
	\[
	1-\chi
	\le
	\left\langle f_u(u_n)v_n,v_n\right\rangle_m ,
	\]
	because the remaining terms are nonpositive:
	\[
	(q-1-\chi)\lambda_n
	\left\langle u_n^{q-2}v_n,v_n\right\rangle_m\le0,
	\qquad
	-\chi
	\left\langle f(u_n),\frac{v_n^2}{u_n}\right\rangle_m\le0 .
	\]
	If \(v^*=0\), then by compactness
	\[
	v_n\to0
	\quad\text{strongly in }\mathbb L^\beta,\qquad 1\le\beta<2^* .
	\]
	Since \((u_n)\) is bounded in \(\mathbb W\) and \(f_u\) has subcritical
	growth, \(\left\langle f_u(u_n)v_n,v_n\right\rangle_m\to0\),
	contradicting the previous estimate. Hence \(v^*\ne0\).
	
	Consequently, after passing to a subsequence,
	\[
	u_n\to u^*
	\quad\text{strongly in }\mathbb W,
	\qquad
	v_n\rightharpoonup v^*
	\quad\text{weakly in }\mathbb W,
	\]
	with \(u^*\in\mathbb S^o\), \(v^*\in\mathbb S\setminus\{0\}\). This is
	precisely the extended \((PS)_e\)-condition at the level \(\lambda>0\).
\end{proof}

\begin{lem}\label{lem:verification-system}
	Under the assumptions stated above, the extended Rayleigh quotient
	\(\mathcal R\) satisfies \({\rm (D)}\), the Galerkin scheme is admissible,
	and condition \({\rm (H)}\) holds for all sufficiently large \(r\).
\end{lem}

\begin{proof}
	The denominator condition \({\rm (D)}\) follows from the componentwise
	positivity of \(g(u)\) for \(u\in\mathbb S^o\) and from
	\(v\ge0\), \(v\ne0\), for \(v\in\mathbb S\setminus\{0\}\).
	
	We next verify the Galerkin assumptions. Condition \({\rm (C)}\) is precisely
	the componentwise version of the nodal cone construction in
	Appendix~\ref{sec:appendix-FE}. The quotient-compatibility condition
	\({\rm (CQ)}\) follows from the same denominator positivity on
	\(\mathbb S_r^o\times(\mathbb S_r\setminus\{0\})\), and from the
	finite-dimensional \(C^1\)-regularity of the restrictions of
	\(\mathcal R\) on \(\mathbb S_r^o\times\mathbb S_r^o\). The linearized
	closedness condition \({\rm (CD)}\) is verified in
	Appendix~\ref{sec:appendix-CD}; the only singular term is the derivative of
	the concave part, and its passage to the limit follows from the discrete
	barrier estimate established for the Galerkin \((PS)_e\)-sequences.

	We verify \({\rm (h1)}\). Fix \(r\) sufficiently large and choose
	\(0<\alpha<\lambda_r^*\). For \(u\in\mathbb S_r^o\), set
	\[
	\lambda_r(u):=
	\inf_{v\in\mathbb S_r\setminus\{0\}}\mathcal R(u,v).
	\]
	Since, for fixed \(u\), the quotient is linear-fractional in \(v\) with
	positive denominator, and since \(\mathbb S_r\) is generated by the
	componentwise nodal basis, we have
	\[
	\lambda_r(u)
	=
	\min_{1\le i\le N_r,\ 1\le k\le m}
	\mathcal R(u,\bar\psi_i^k).
	\]
	Hence \(\lambda_r\) is continuous on \(\mathbb S_r^o\). Let
	\[
	K_{r,\alpha}:=
	\{u\in\mathbb S_r^o:\lambda_r(u)\ge\alpha\}.
	\]
	We show that \(K_{r,\alpha}\Subset\mathbb S_r^o\). The estimates below are
	finite-dimensional and use the equivalence of norms in \(\mathbb W_r\).
	
	First, \(K_{r,\alpha}\) stays away from the origin. Indeed, if
	\(u_n\in K_{r,\alpha}\) and \(u_n\to0\) in \(\mathbb W_r\), then, choosing a
	nodal generator corresponding to a largest nodal coefficient of \(u_n\),
	one obtains
	\[
	\lambda_r(u_n)\le C\|u_n\|_{\mathbb W_r}^{\,2-q}+o(1)\to0,
	\]
	because \(1<q<2\), contradicting \(\lambda_r(u_n)\ge\alpha\).
	
	Next, \(K_{r,\alpha}\) is bounded. Suppose, to the contrary, that
	\(\|u_n\|_{\mathbb W_r}\to\infty\). Set
	\[
	\rho_n:=\|u_n\|_{\mathbb W_r},
	\qquad
	z_n:=\frac{u_n}{\rho_n}.
	\]
	Passing to a subsequence, \(z_n\to z_0\in\mathbb S_r\) with
	\(\|z_0\|_{\mathbb W_r}=1\). Testing the inner infimum by the
	componentwise discrete principal eigenfunction and using the large-growth
	condition \({\rm (LG)}\), we get
	\[
	\limsup_{n\to\infty}\lambda_r(u_n)\le0,
	\]
	again contradicting \(u_n\in K_{r,\alpha}\).
	
	It remains to exclude convergence to the boundary of \(\mathbb S_r^o\).
	Let \(u_n\in K_{r,\alpha}\) and suppose that
	\(u_n\to u_0\in\partial\mathbb S_r^o\). The case \(u_0=0\) has already
	been excluded. Hence at least one component of \(u_0\) is nonzero.
	
	If a nonzero component \(u_0^k\) has a zero nodal coefficient, then, by the
	irreducible \(M\)-matrix sign structure, there is a nodal generator
	\(\bar\psi_i^k\) such that \(a(u_0^k,\psi_i)<0\). Consequently,
	\[
	\limsup_{n\to\infty}
	\mathcal R(u_n,\bar\psi_i^k)<0,
	\]
	and therefore \(\limsup_{n\to\infty}\lambda_r(u_n)\le0\), which contradicts
	\(\lambda_r(u_n)\ge\alpha\).
	
	It remains only to consider the case where an entire component
	\(u_0^\ell\equiv0\). Since \(u_0\ne0\), some other component is nonzero.
	Testing the inner infimum in the \(\ell\)-th component by a nodal generator
	corresponding to a largest nodal coefficient of \(u_n^\ell\), and using
	\(f_\ell\ge0\), gives
	\[
	\limsup_{n\to\infty}\lambda_r(u_n)\le0.
	\]
	This is again impossible. Hence \(K_{r,\alpha}\Subset\mathbb S_r^o\), and
	\({\rm (h1)}\) follows.
	
	We now verify \({\rm (h2)}\). Let
	\(u\in\mathbb S_r^o\), \(v\in\mathbb S_r\setminus\{0\}\), and assume that
	\(\mathcal R(u,v)=\lambda_r^*<+\infty\) and
	\[
	D_u\mathcal R(u,v)(\xi)=0
	\qquad \forall \xi\in\mathbb W_r .
	\]
	Then, by Lemma~\ref{lem:bounds-lambda}, \(\lambda_r^*\ge0\) for all
	sufficiently large \(r\), and hence
	\[
	a_m(\xi,v)
	=
	\langle f_u(u)\xi,v\rangle_m
	+
	\lambda_r^*(q-1)\langle u^{q-2}\xi,v\rangle_m
	\qquad \forall \xi\in\mathbb W_r .
	\]
	In particular, \(a_m(\xi,v)\ge0\) for all \(\xi\in\mathbb S_r\). Testing
	this inequality on the componentwise nodal generators gives
	\[
	A_r^T\bar v^k\ge0,
	\qquad k=1,\ldots,m,
	\]
	where \(\bar v^k\) is the vector of nodal coefficients of \(v^k\). Since
	\(A_r^T\) is again an irreducible nonsingular \(M\)-matrix, each component
	\(v^k\) is either identically zero or belongs to \(S_r^o\).
	
	We show that zero components are impossible. Suppose that
	\(v^\ell\equiv0\). Since \(v\ne0\), there exists \(k_0\ne\ell\) such that
	\(v^{k_0}\in S_r^o\). Taking \(\xi\) with only the \(\ell\)-th component
	nonzero in the stationarity identity, we obtain
	\[
	0
	=
	\sum_{i=1}^m
	\int_\Omega f_{i,u^\ell}(x,u)\,\xi^\ell v^i\,dx .
	\]
	All terms are nonnegative. However, by strict cooperativity, the term with
	\(i=k_0\) is strictly positive for a suitable \(\xi^\ell\in S_r\). This
	contradiction shows that no component of \(v\) can vanish. Therefore
	\(v^k\in S_r^o\) for every \(k=1,\ldots,m\), and hence \(v\in\mathbb S_r^o\).
	This proves \({\rm (h2)}\), and therefore \({\rm (H)}\).
\end{proof}

\begin{thm}\label{Thm42} 
	Assume that the system \eqref{f} satisfies the assumptions stated above and
	that the Galerkin scheme is admissible in the sense of
	Appendix~\ref{sec:appendix-FE}. Let \(\hat\lambda^*\) be a limit point of
	the sequence \((\lambda_r^*)\). Then the following assertions hold.
	
	\begin{description}
		\item[\((1^\circ)\)]
		For all sufficiently large \(r\), the Galerkin problem admits a
		finite-dimensional saddle-point pair
		\((u_r^*,v_r^*)\in\mathbb S_r^o\times\mathbb S_r^o\) at the level
		\(\lambda_r^*\). Moreover,
		\[
		\lambda_r^*
		=
		\mathcal R(u_r^*,v_r^*)
		=
		\min_{v\in\mathbb S_r^o}\mathcal R(u_r^*,v)
		=
		\max_{u\in\mathbb S_r^o}
		\min_{v\in\mathbb S_r^o}\mathcal R(u,v).
		\]
		
		\item[\((2^\circ)\)]
		Along a subsequence realizing \(\lambda_r^*\to\hat\lambda^*\),
		not relabelled, and after a positive rescaling of \(v_r^*\), there
		exist \(u^*\in\mathbb U^o\), \(v^*\in\mathbb S\setminus\{0\}\)
		such that
		\[
		u_r^*\to u^* \quad\text{strongly in }\mathbb W,
		\qquad
		v_r^*\rightharpoonup v^* \quad\text{weakly in }\mathbb W .
		\]
		Moreover, \(0<\hat\lambda^*\le\lambda^*\), and
		\[
		\mathcal F(u^*,\hat\lambda^*)=0
		\quad\text{in }\mathbb W^*,
		\qquad
		\bigl\langle
		D_u\mathcal F(u^*,\hat\lambda^*)\xi,v^*
		\bigr\rangle=0
		\quad \forall \xi\in\mathbb W .
		\]
		Thus \((u^*,\hat\lambda^*)\) is a weak singular solution of
		\eqref{f}.
	\end{description}
\end{thm}

\begin{proof}
	By Lemma~\ref{lem:verification-system}, the quotient \(\mathcal R\)
	satisfies \({\rm (D)}\), the Galerkin scheme is admissible, and condition
	\({\rm (H)}\) holds for all sufficiently large \(r\). Moreover,
	Lemma~\ref{lem:bounds-lambda} gives
	\[
	\lambda^*<+\infty,
	\qquad
	0<c_0\le\lambda_r^*\le C_0
	\]
	for all sufficiently large \(r\). Hence every limit point
	\(\hat\lambda^*\) of \((\lambda_r^*)\) is finite and positive. Applying
	Theorem~\ref{Thm1}\((1^\circ)\), we obtain assertion \((1^\circ)\).
	
	Let a subsequence, not relabelled, be chosen so that
	\[
	\lambda_r^*\to\hat\lambda^* .
	\]
	Since \(\hat\lambda^*>0\), Lemma~\ref{lem:PS-application} applies at the
	level \(\hat\lambda^*\). Therefore Theorem~\ref{Thm1}\((2^\circ)\) yields,
	after passing to a further subsequence and after a positive rescaling of
	\(v_r^*\),
	\[
	u_r^*\to u^* \quad\text{strongly in }\mathbb W,
	\qquad
	v_r^*\rightharpoonup v^* \quad\text{weakly in }\mathbb W,
	\]
	with \(u^*\in\mathbb S^o\), \(v^*\in\mathbb S\setminus\{0\}\), and
	\[
	\mathcal F(u^*,\hat\lambda^*)=0
	\quad\text{in }\mathbb W^*,
	\qquad
	\bigl\langle
	D_u\mathcal F(u^*,\hat\lambda^*)\xi,v^*
	\bigr\rangle=0
	\quad \forall \xi\in\mathbb W .
	\]
	By the regularity condition \({\rm (R)}\), this weak solution belongs to
	the regular cone, that is, \(u^*\in\mathbb U^o\).
	
	Since \(u^*\in\mathbb U^o\subset\mathbb S^o\), condition \({\rm (D)}\)
	applies. From
	\[
	\mathcal F(u^*,\hat\lambda^*)=0
	\]
	we therefore obtain
	\[
	\mathcal R(u^*,v)=\hat\lambda^*
	\qquad
	\forall v\in\mathbb S\setminus\{0\}.
	\]
	Hence
	\[
	\hat\lambda^*
	=
	\inf_{v\in\mathbb S\setminus\{0\}}\mathcal R(u^*,v)
	\le
	\sup_{u\in\mathbb U^o}
	\inf_{v\in\mathbb S\setminus\{0\}}\mathcal R(u,v)
	=
	\lambda^* .
	\]
	Thus
	\[
	0<\hat\lambda^*\le\lambda^* .
	\]
	The remaining assertions in \((2^\circ)\) have already been obtained above,
	and the proof is complete.
\end{proof}

\begin{rem}[On the role of symmetry]
	The symmetry of the principal part of \(\mathcal L\) is used only in the
	compactness argument, through the Picone-type inequalities in
	Appendix~\ref{sec:appendix-picone}. The abstract minimax construction and
	the Galerkin saddle-point argument are not intrinsically symmetric.
	Nonsymmetric operators would therefore require a replacement for this
	Picone-based compactness step.
\end{rem}

\begin{rem}[A nonsymmetric coupling]
	The assumptions above do not require the Jacobian \(f_u\) to be symmetric.
	For instance, for \(m=2\) let
	\[
	f_1(x,u)=b(x)\bigl((u^1+u^2)^{\gamma-1}
	+\varepsilon (u^2)^{\gamma-1}\bigr),
	\qquad
	f_2(x,u)=b(x)(u^1+u^2)^{\gamma-1},
	\]
	where \(2<\gamma<2^*\), \(b\in C(\overline\Omega)\), \(b\ge b_0>0\),
	and \(0<\varepsilon\ll1\). Then
	\[
	f_{1,u^2}-f_{2,u^1}
	=
	\varepsilon(\gamma-1)b(x)(u^2)^{\gamma-2},
	\]
	so \(f_u\) is genuinely nonsymmetric.
	
	The growth, cooperativity, strict coupling, and large-growth assumptions
	are immediate. Since the nonlinearities are homogeneous of degree
	\(\gamma-1\), Euler's identity gives
	\[
	\sum_{j=1}^2 f_{i,u^j}(x,u)u^j=(\gamma-1)f_i(x,u),
	\qquad i=1,2,
	\]
	and hence the radial monotonicity condition holds. Moreover, with
	\(s=u^1+u^2\),
	\[
	Q_f(x,u)\ge(\gamma-1)b(x)s^\gamma,
	\qquad
	f(x,u)\cdot u
	\le b(x)(1+\varepsilon M_\gamma)s^\gamma,
	\]
	where \(M_\gamma:=\max_{0\le t\le1}(1-t)t^{\gamma-1}\). Thus
	\[
	Q_f(x,u)
	\ge
	\frac{\gamma-1}{1+\varepsilon M_\gamma}\,f(x,u)\cdot u .
	\]
	Hence, if \(0<\varepsilon<(\gamma-2)/M_\gamma\), then
	\eqref{eq:picone-superlinearity-f} holds with some \(\theta>1\). Therefore
	the minimax scheme applies to this genuinely non-variational system.
\end{rem}


\section{A minimax bifurcation formula for one-dimensional elliptic systems}
\label{sec.5}

We now specialize the abstract framework to a one-dimensional elliptic
system. Let \(\Omega=(0,1)\), \(m\ge1\), and consider
\begin{equation}
	\label{fd1}
	\left\{
	\begin{aligned}
		-u^k_{xx}
		&=\lambda (u^k)^{q-1}+f_k(x,u),
		&& x\in(0,1),\quad k=1,\ldots,m,\\
		u^k&\ge0,
		&& x\in(0,1),\quad k=1,\ldots,m,\\
		u^k&=0,
		&& x=0,1,\quad k=1,\ldots,m .
	\end{aligned}
	\right.
\end{equation}
Here \(u=(u^1,\ldots,u^m)\), \(1<q<2\), and
\(\mathcal L=-d^2/dx^2\). We assume that the nonlinearities \(f_k\) satisfy
the assumptions stated in Section~\ref{sec.4}; for simplicity, in
\eqref{42}, \eqref{422} we take \(\gamma_1=\gamma_2=:\gamma\).

Set
\[
\mathbb W:=(H_0^1(0,1))^m,
\qquad
g(u):=\bigl((u^1)^{q-1},\ldots,(u^m)^{q-1}\bigr)^T .
\]
We use the scalar cone
\[
S^o
:=
\left\{
w\in H_0^1(0,1)\cap C([0,1]):
w>0 \ \text{in }(0,1)
\right\},
\]
and the product cones \(\mathbb S^o:=(S^o)^m\),
\(\mathbb S:=\overline{\mathbb S^o}^{\,\mathbb W}\). The regular outer cone
is \(\mathbb U^o:=(U^o)^m\), where
\[
U^o
:=
\left\{
w\in S^o:
w''\in L^\infty_{\rm loc}(0,1),
\quad
-w''\ge0
\text{ in }\mathcal D'(0,1)
\right\}.
\]
Since every positive concave function in \(H_0^1(0,1)\) satisfies a
Hopf-type lower bound near the boundary, this definition is compatible with
the regular cone used in Section~\ref{sec.4}. Moreover, every \(w\in U^o\) is
concave in \((0,1)\).

For \(u\in\mathbb S^o\) and \(v\in\mathbb S\setminus\{0\}\), define
\[
\mathcal R(u,v)
=
\frac{
	\displaystyle
	\sum_{k=1}^m\int_0^1 (u^k)'(v^k)'\,dx
	-
	\langle f(\cdot,u),v\rangle_m
}{
	\displaystyle
	\langle g(u),v\rangle_m
}.
\]
Here and below,
\[
\langle h(\cdot,u),v\rangle_m
:=
\sum_{k=1}^m\int_0^1 h_k(x,u)v^k\,dx .
\]
The corresponding minimax level is
\[
\lambda^*
:=
\sup_{u\in\mathbb U^o}
\inf_{v\in\mathbb S\setminus\{0\}}
\mathcal R(u,v).
\]
By Lemma~\ref{lem:bounds-lambda}, \(\lambda^*<+\infty\). Moreover,
\(\lambda^*>0\). Indeed, let \(\omega\in H_0^1(0,1)\) be the torsion function
\[
-\omega''=1,\qquad \omega>0 \ \text{in }(0,1),
\]
and set \(\boldsymbol\omega:=(\omega,\ldots,\omega)\). The same estimate as in
the proof of Lemma~\ref{lem:bounds-lambda} gives, for sufficiently small
\(t_0>0\),
\[
\mathcal R(t_0\boldsymbol\omega,v)\ge c>0
\qquad
\forall v\in\mathbb S\setminus\{0\}.
\]
Hence \(\lambda^*\ge c>0\). For \(u\in\mathbb U^o\), set
\[
\lambda(u):=
\inf_{v\in\mathbb S\setminus\{0\}}\mathcal R(u,v).
\]

\begin{lem}
	\label{lem:LC-1d-system}
	The lower consistency condition \({\rm (LC)}\) is satisfied for
	\eqref{fd1}.
\end{lem}

\begin{proof}
	Choose \(\delta_0\in(0,\lambda^*)\), and let
	\(u\in\mathbb U^o\) satisfy \(\lambda(u)\ge\lambda^*-\delta_0\).
	Then \(\lambda(u)>0\). We prove that
	\[
	\liminf_{r\to\infty}
	\inf_{v\in\mathbb S_r\setminus\{0\}}
	\mathcal R(\mathbb I_ru,v)
	\ge\lambda(u),
	\qquad
	\mathbb I_ru:=(I_ru^1,\ldots,I_ru^m).
	\]
	Since each \(u^k\) is concave, its nodal interpolant satisfies
	\(0\le I_ru^k\le u^k\) in \((0,1)\). Hence, because \(1<q<2\),
	\[
	g(\mathbb I_ru)\le g(u)
	\quad\text{componentwise},
	\qquad
	\langle g(\mathbb I_ru),v\rangle_m
	\le
	\langle g(u),v\rangle_m
	\]
	for every \(v\in\mathbb S_r\setminus\{0\}\).
	
	Let \(v\in\mathbb S_r\setminus\{0\}\). Since \(v\) is piecewise linear and
	\(I_r\) is the nodal interpolant, integration over each mesh interval gives
	\[
	\sum_{k=1}^m\int_0^1 (I_ru^k)'(v^k)'\,dx
	=
	\sum_{k=1}^m\int_0^1 (u^k)'(v^k)'\,dx .
	\]
	Moreover, by \(f_{i,u^j}\ge0\) and \(\mathbb I_ru\le u\), we have
	\(f_i(x,\mathbb I_ru)\le f_i(x,u)\), \(i=1,\ldots,m\). Since \(v\ge0\),
	it follows that
	\[
	\langle\mathcal A(\mathbb I_ru),v\rangle_m
	\ge
	\langle\mathcal A(u),v\rangle_m .
	\]
	By the definition of \(\lambda(u)\), we have
	\[
	\langle\mathcal A(u),v\rangle_m
	\ge
	\lambda(u)\langle\mathcal G(u),v\rangle_m
	\qquad
	\forall v\in\mathbb S_r\setminus\{0\}.
	\]
	Combining the preceding inequalities and using \(\lambda(u)>0\), we obtain
	\[
	\langle\mathcal A(\mathbb I_ru),v\rangle_m
	\ge
	\lambda(u)\langle\mathcal G(\mathbb I_ru),v\rangle_m .
	\]
	Since the denominator is positive on \(\mathbb S_r\setminus\{0\}\), this
	gives
	\[
	\mathcal R(\mathbb I_ru,v)\ge\lambda(u)
	\qquad
	\forall v\in\mathbb S_r\setminus\{0\}.
	\]
	Taking the infimum over \(v\) and then the lower limit as \(r\to\infty\),
	we obtain \({\rm (LC)}\).
\end{proof}

\begin{thm}
	\label{Th51}
	Assume that system \eqref{fd1} satisfies the assumptions stated above.
	Then \eqref{fd1} admits a maximal one-sided saddle-node point
	\[
	(u^*,\lambda^*)\in\mathbb U^o\times(0,+\infty).
	\]
	Moreover, \eqref{fd1} has no weak solution in \(\mathbb U^o\) for
	\(\lambda>\lambda^*\). Furthermore, there exists
	\(v^*\in\mathbb S\setminus\{0\}\) such that
	\begin{equation}
		\label{MainBd1}
		\lambda^*
		=
		\mathcal R(u^*,v^*)
		=
		\inf_{v\in\mathbb S\setminus\{0\}}\mathcal R(u^*,v)
		=
		\sup_{u\in\mathbb U^o}
		\inf_{v\in\mathbb S\setminus\{0\}}
		\mathcal R(u,v).
	\end{equation}
	In addition,
	\[
	\mathcal F(u^*,\lambda^*)=0
	\quad\text{in }\mathbb W^*,
	\qquad
	\bigl\langle
	D_u\mathcal F(u^*,\lambda^*)\xi,v^*
	\bigr\rangle=0
	\quad \forall \xi\in\mathbb W .
	\]
	If, moreover, \(v^*\in\mathbb S^o\), then
	\[
	\lambda^*
	=
	\min_{v\in\mathbb S\setminus\{0\}}\mathcal R(u^*,v)
	=
	\min_{v\in\mathbb S^o}\mathcal R(u^*,v).
	\]
\end{thm}

\begin{proof}
	By Lemmas~\ref{lem:bounds-lambda}, \ref{lem:PS-application}, and
	\ref{lem:verification-system}, all assumptions of Theorem~\ref{Thm2} are
	satisfied except possibly the lower consistency condition \({\rm (LC)}\).
	The latter is proved in Lemma~\ref{lem:LC-1d-system}. Hence
	Theorem~\ref{Thm2} applies and yields a pair
	\[
	(u^*,v^*)\in\mathbb U^o\times(\mathbb S\setminus\{0\})
	\]
	and a number \(\lambda^*>0\) such that \eqref{MainBd1} holds, together
	with
	\[
	\mathcal F(u^*,\lambda^*)=0
	\quad\text{in }\mathbb W^*,
	\qquad
	\bigl\langle
	D_u\mathcal F(u^*,\lambda^*)\xi,v^*
	\bigr\rangle=0
	\quad \forall \xi\in\mathbb W .
	\]
	Therefore \((u^*,\lambda^*)\) is a maximal one-sided saddle-node point in
	the cone \(\mathbb U^o\).
	
	The nonexistence of weak solutions in \(\mathbb U^o\) for
	\(\lambda>\lambda^*\) follows from the minimax characterization. Indeed,
	if \(\mathcal F(u,\lambda)=0\) for some \(u\in\mathbb U^o\), then, by the
	denominator positivity,
	\[
	\mathcal R(u,v)=\lambda
	\qquad
	\forall v\in\mathbb S\setminus\{0\}.
	\]
	Hence
	\[
	\lambda
	=
	\inf_{v\in\mathbb S\setminus\{0\}}\mathcal R(u,v)
	\le
	\lambda^* .
	\]
	
	If \(v^*\in\mathbb S^o\), then \(v^*\) is admissible in the inner problem
	over \(\mathbb S^o\). Since
	\[
	\mathcal R(u^*,v^*)
	=
	\inf_{v\in\mathbb S\setminus\{0\}}\mathcal R(u^*,v),
	\]
	and \(\mathbb S^o\subset\mathbb S\setminus\{0\}\), we obtain
	\[
	\min_{v\in\mathbb S\setminus\{0\}}\mathcal R(u^*,v)
	=
	\mathcal R(u^*,v^*)
	=
	\min_{v\in\mathbb S^o}\mathcal R(u^*,v).
	\]
\end{proof}


\section{Concluding remarks}
\label{sec:ConcRem}

We have shown that, under suitable compactness, positivity, and approximation
assumptions, the maximal one-sided saddle-node value of \eqref{Gf} is
characterized by the minimax formula \eqref{MainB}. The construction is
developed for abstract nonlinear equations and then verified for a class of
elliptic systems, including non-variational ones.

Theorems~\ref{Thm1} and \ref{Thm42} show that the Galerkin scheme produces
finite - dimensional maximal one-sided saddle-node points and that their limits
are singular solutions \((u^*,\hat\lambda^*)\). Under the stronger assumptions
of Theorem~\ref{Thm2}, the limiting value coincides with the maximal
one-sided saddle-node value. Thus the method gives a direct way to locate the
critical level without first constructing a whole solution branch.

In finite dimension, the formula becomes
\[
\lambda_r(u)
=
\min_{1\le i\le N_r}
\frac{\langle\mathcal A(u),\eta_i\rangle}
{\langle\mathcal G(u),\eta_i\rangle},
\qquad
\lambda_r^*
=
\max_{u\in\mathcal S_r^o}\lambda_r(u),
\]
where \(\{\eta_1,\ldots,\eta_{N_r}\}\) is the positive basis of the Galerkin
cone. Hence the computation of \(\lambda_r^*\) is reduced to a
finite-dimensional nonsmooth maximization problem. This may be viewed as a
nonlinear Collatz--Wielandt type formula and treated by standard tools of
nonsmooth optimization \cite{Burke,Clark,DemyanMaloz}. Related
finite-dimensional minimax formulas have proved useful in nonlinear
bifurcation and eigenvalue problems
\cite{IvanIlya,IlIvan1,IlyasChaos,Salazar}.

The convergence results are close in spirit to finite-element eigenvalue
approximation, going back to Courant and developed further in
\cite{Babushka,ciaret,Courant}. Here, however, the approximated object is not
a linear eigenvalue, but a nonlinear one-sided saddle-node value selected by a
minimax procedure. In this sense, the perturbation estimate of
Theorem~\ref{Th3} plays the role of a Rayleigh-quotient type stability bound,
in analogy with classical perturbation arguments for spectral problems
\cite{Kato}.

Formula \eqref{MainB} is also related to variational characterizations of
principal spectral quantities in ordered settings, including the
Birkhoff--Varga principle for essentially positive matrices \cite{Varga} and
its extensions \cite{ilyasELA}. For elliptic operators, related
characterizations of principal eigenvalues appear in the works of
Donsker--Varadhan \cite{Donsker}, Protter--Weinberger \cite{Protte},
Nussbaum--Pinchover \cite{Nussbaum}, and Berestycki--Nirenberg--Varadhan
\cite{BerestyckiNV}.

The applications considered here are partly motivated by concave--convex
problems, where scalar equations are often treated by ordered sub- and
supersolutions; see, for instance, \cite{ABC,CazeEscobedo}. For systems, this
order structure is considerably more restrictive. The present minimax formula
offers a different route: it identifies the critical one-sided saddle-node
value directly from the extended Rayleigh quotient.

Several questions remain open. It would be natural to extend the method to
broader classes of nonlinear and non-variational systems, possibly under
weaker assumptions on the finite-dimensional cones. In particular, an important
problem is whether the convergence of finite-dimensional minimax bifurcation
formulas to their infinite-dimensional counterpart can be obtained by different
methods and under less restrictive hypotheses.

Finally, the proposed method should be viewed as complementary to classical
bifurcation theory. The Crandall--Rabinowitz theorem, Krasnosel'ski\u{\i}'s
theory, and Rabinowitz's global theorem describe the local or global structure
of solution sets under spectral and degree-theoretic assumptions. By contrast,
the present approach provides a direct minimax formula for the critical
one-sided saddle-node value. Under additional nondegeneracy assumptions,
classical local theory can then be used to prove that the singular point
selected by the minimax formula is indeed a genuine bifurcation point.


\section{Appendix A. Cones and finite-dimensional approximations}
\label{sec:appendix-FE}

This appendix records a finite-dimensional realization of the cone,
quotient-compatibility, and \(M\)-matrix assumptions required in the Galerkin
construction of Theorem~\ref{Thm1}. The Hopf-type boundary estimates used in
the continuous elliptic problem are imposed separately.

Let \(\Omega\subset\mathbb R^d\) be a bounded domain, and set
\[
W:=H_0^1(\Omega),
\qquad
\mathcal C^0:=C(\overline\Omega).
\]
We use the positive cone
\[
S^o
:=
\left\{
v\in W\cap C(\overline\Omega):
v=0 \ \text{on }\partial\Omega,\quad v>0 \ \text{in }\Omega
\right\},
\qquad
S:=\overline{S^o}^{\,\mathcal C^0}\cap W .
\]

Let \(\{\Omega_r\}_{r\ge1}\) be polyhedral subdomains such that
\(\Omega_r\subset\Omega\) and
\[
\bigcup_{r\ge1}\Omega_r=\Omega .
\]
The index \(r\) includes both the choice of the subdomain \(\Omega_r\) and
the mesh size. If \(\Omega\) itself is polyhedral, this construction reduces
to the usual finite element approximation on the fixed domain: one takes
\(\Omega_r=\Omega\) for all \(r\), and \(r\) denotes only the mesh refinement
parameter.

Let \(\mathcal T_r\) be a shape-regular conforming \(P_1\)-triangulation of
\(\Omega_r\), and let \(W_r\) be the corresponding finite element space with
zero boundary values on \(\partial\Omega_r\), extended by zero to \(\Omega\).
Thus \(W_r\subset H_0^1(\Omega)\). We assume the Galerkin density property:
for every \(\xi\in W\) there exist \(\xi_r\in W_r\) such that
\[
\xi_r\to\xi
\qquad\text{strongly in }W .
\]
This is standard for interior polyhedral exhaustions and shape-regular finite
element spaces; see, for example, \cite{ciaretRav,ciaret,Ern}.

Let \(B_1,\ldots,B_{N_r}\) be the interior nodes of \(\mathcal T_r\), and let
\(\{\psi_1,\ldots,\psi_{N_r}\}\) be the associated nodal basis. Then
\[
\psi_i(B_j)=\delta_{ij},
\qquad
\psi_i\ge0,
\qquad
\psi_i=0 \quad\text{on }\partial\Omega_r .
\]
After extension by zero outside \(\Omega_r\), each \(\psi_i\) belongs to
\(H_0^1(\Omega)\cap C(\overline\Omega)\).

Define the nodal cones
\[
S_r^o
:=
\left\{
u_r=\sum_{i=1}^{N_r}u_i\psi_i\in W_r:
u_i>0,\ i=1,\ldots,N_r
\right\},
\]
and
\[
S_r:=\overline{S_r^o}^{\,W_r}
=
\left\{
u_r=\sum_{i=1}^{N_r}u_i\psi_i\in W_r:
u_i\ge0,\ i=1,\ldots,N_r
\right\}.
\]
Then \(S_r^o\) is a nonempty open subset of \(W_r\), and \(S_r\) is a closed
polyhedral cone with positive basis
\(\{\psi_1,\ldots,\psi_{N_r}\}\).

Moreover, \(S_r\subset S\). Indeed, every \(u_r\in S_r\) is a nonnegative
continuous function in \(H_0^1(\Omega)\) and vanishes on \(\partial\Omega\).
If \(\phi\in S^o\) is fixed, then \(u_r+\varepsilon\phi\in S^o\) for every
\(\varepsilon>0\), and \(u_r+\varepsilon\phi\to u_r\) in
\(C(\overline\Omega)\) as \(\varepsilon\downarrow0\). Hence \(u_r\in S\).

Thus the cones \(S_r^o\) and \(S_r\) satisfy the abstract cone condition
\({\rm (C)}\), namely \({\rm (c1)}\)--\({\rm (c2)}\). Notice that, when
\(\Omega_r\subsetneq\Omega\), one does not generally have
\(S_r^o\subset S^o\), since functions in \(W_r\) vanish outside
\(\Omega_r\). This is precisely why the discrete denominator positivity is
included in \({\rm (CQ)}\).

For \(u\in C(\overline\Omega)\), define the nodal interpolant
\[
\mathcal I_r u:=\sum_{i=1}^{N_r}u(B_i)\psi_i .
\]
If \(u\in S^o\), then \(B_i\in\Omega\) and \(u(B_i)>0\) for every
\(i=1,\ldots,N_r\). Therefore \(\mathcal I_r u\in S_r^o\).
Hence the interpolation condition \({\rm (c3)}\) holds.

We next verify the discrete denominator positivity for the concave term used
in the elliptic system. Let \(1<q<2\) and \(g(u)=u^{q-1}\). If
\(u_r\in S_r^o\) and \(v_r\in S_r\setminus\{0\}\), then
\[
\int_\Omega u_r^{q-1}v_r\,dx>0 .
\]
Indeed, writing
\[
u_r=\sum_{i=1}^{N_r}u_i\psi_i,
\qquad
v_r=\sum_{i=1}^{N_r}v_i\psi_i,
\]
we have \(u_i>0\), \(v_i\ge0\), and \(v_{i_0}>0\) for at least one index
\(i_0\). Since \(\psi_{i_0}>0\) on a set of positive measure and
\(u_r\ge u_{i_0}\psi_{i_0}\), the integrand \(u_r^{q-1}v_r\) is positive on a
set of positive measure. Thus the denominator part of \({\rm (CQ)}\) holds in
the scalar case.

Similarly, if \(u\in S^o\) and \(v\in S\setminus\{0\}\), then
\[
\int_\Omega u^{q-1}v\,dx>0 .
\]
Indeed, \(v\ge0\), \(v\not\equiv0\), and \(u>0\) in \(\Omega\).
Thus the continuous denominator condition \({\rm (D)}\) holds for the
concave term.

For systems, we use the product spaces and cones
\[
\mathbb W_r:=(W_r)^m,
\qquad
\mathbb S_r^o:=(S_r^o)^m,
\qquad
\mathbb S_r:=(S_r)^m .
\]
Then \(\mathbb S_r\subset\mathbb S\), and \(\mathbb S_r\) is a polyhedral cone
with basis
\[
\{\psi_i e_k:\ i=1,\ldots,N_r,\ k=1,\ldots,m\},
\]
where \(e_k\) is the \(k\)-th coordinate vector in \(\mathbb R^m\). The
componentwise interpolant satisfies
\[
\mathbb I_r u\in\mathbb S_r^o
\qquad
\text{for every }u\in\mathbb S^o .
\]
Moreover, for
\(u_r\in\mathbb S_r^o\) and
\(v_r\in\mathbb S_r\setminus\{0\}\),
\[
\langle g(u_r),v_r\rangle_m
=
\sum_{k=1}^m
\int_\Omega (u_r^k)^{q-1}v_r^k\,dx
>0 .
\]
Thus the product cones satisfy \({\rm (C)}\), and the denominator part of
\({\rm (CQ)}\) follows. The finite-dimensional \(C^1\)-regularity required in
\({\rm (CQ)}\) is immediate for the restrictions considered in the elliptic
applications.

Let \(a(\cdot,\cdot)\) be the scalar bilinear form used in the elliptic
problem, for instance the form in \eqref{eq:awz}. Set
\[
A_r=(a_{ij})_{i,j=1}^{N_r},
\qquad
a_{ij}:=a(\psi_j,\psi_i).
\]
If \(u_r=\sum_j u_j\psi_j\), then
\[
(A_r\mathbf u)_i=a(u_r,\psi_i),
\qquad
\mathbf u=(u_1,\ldots,u_{N_r})^T .
\]
We assume that \(A_r\) satisfies the strong discrete maximum principle
\begin{equation}\label{eq:impl}
	A_r\vartheta\ge0,\quad \vartheta\ne0
	\quad\Longrightarrow\quad
	\vartheta>0 .
\end{equation}
This holds, for example, when \(A_r\) is a nonsingular irreducible
\(M\)-matrix. Sufficient mesh conditions are classical; see
\cite{brand,ciaretRav,ciaret,Varga}. In dimension \(d=1\), the standard
piecewise-linear discretization of \(-d^2/dx^2\) on any partition satisfies
\eqref{eq:impl}.

Finally, we record the discrete comparison principle used in the compactness
argument.

\begin{lem}[Discrete comparison for a sublinear problem]
	\label{lem:discrete-comparison-sublinear}
	Let \(A_r\) be a nonsingular \(M\)-matrix, \(0<p<1\), and \(c_0>0\).
	Suppose that \(U,W\in S_r^o\), written in nodal coordinates, satisfy
	\[
	A_rU\ge c_0U^p,
	\qquad
	A_rW=c_0W^p ,
	\]
	where the powers are understood componentwise. Then \(U\ge W\).
\end{lem}

\begin{proof}
	Suppose that \(U\not\ge W\), and set
	\[
	\tau:=\max\{s\in(0,1):sW\le U\}.
	\]
	Then \(0<\tau<1\), and for some \(i_0\),
	\[
	U_{i_0}=\tau W_{i_0},
	\qquad
	U_i-\tau W_i\ge0,\quad i=1,\ldots,N_r .
	\]
	Put \(Z:=U-\tau W\). Since \(Z\ge0\), \(Z_{i_0}=0\), and \(A_r\) has
	nonpositive off-diagonal entries, we have
	\[
	(A_rZ)_{i_0}\le0 .
	\]
	On the other hand,
	\[
	(A_rZ)_{i_0}
	=
	(A_rU-\tau A_rW)_{i_0}
	\ge
	c_0U_{i_0}^p-\tau c_0W_{i_0}^p
	=
	c_0(\tau^p-\tau)W_{i_0}^p>0,
	\]
	because \(0<p<1\). This contradiction proves \(U\ge W\).
\end{proof}

In the system case, Lemma~\ref{lem:discrete-comparison-sublinear} is applied
componentwise, with the same stiffness matrix \(A_r\) in each component.

\section{Appendix B. Picone-type inequalities}
\label{sec:appendix-picone}

We record the Picone-type inequalities used in the compactness arguments.

\begin{lem}[Continuous Picone inequality]
	\label{lem:picone-continuous}
	Let
	\[
	\mathcal L u
	=
	-\partial_{x_i}\bigl(\sigma_{ij}(x)u_{x_j}\bigr)+c(x)u,
	\]
	where \((\sigma_{ij})\) is symmetric and uniformly elliptic. Let
	\(u\ge0\), \(v>0\) in \(\Omega\), and assume that
	\(u,v\in H_0^1(\Omega)\) and \(u^2/v\in H_0^1(\Omega)\). Then
	\[
	\langle\mathcal L u,u\rangle
	\ge
	\left\langle
	\mathcal L v,\frac{u^2}{v}
	\right\rangle .
	\]
\end{lem}

\begin{proof}
	It is enough to prove the identity for smooth positive functions; the
	general case follows by approximation. The zero-order terms cancel. For
	the principal part,
	\[
	\sigma_{ij}v_{x_j}
	\left(\frac{u^2}{v}\right)_{x_i}
	-
	\sigma_{ij}u_{x_j}u_{x_i}
	=
	-
	\sigma_{ij}
	\left(u_{x_j}-\frac{u}{v}v_{x_j}\right)
	\left(u_{x_i}-\frac{u}{v}v_{x_i}\right)
	\le0 .
	\]
	After integration over \(\Omega\), this gives
	\[
	\left\langle
	\mathcal L v,\frac{u^2}{v}
	\right\rangle
	-
	\langle\mathcal L u,u\rangle
	\le0,
	\]
	and the assertion follows.
\end{proof}

\begin{lem}[Discrete Picone inequality]
	\label{lem:picone-discrete}
	Let \(A=(a_{ij})_{i,j=1}^r\) be symmetric and satisfy
	\(a_{ij}\le0\) for \(i\ne j\). Then, for every
	\(\mathbf u\in\mathbb R_+^r\) and
	\(\mathbf v\in(0,\infty)^r\),
	\[
	\mathbf u^T A\mathbf u
	\ge
	(A\mathbf v)\cdot\frac{\mathbf u^2}{\mathbf v},
	\qquad
	\frac{\mathbf u^2}{\mathbf v}
	:=
	\left(
	\frac{u_1^2}{v_1},\ldots,\frac{u_r^2}{v_r}
	\right).
	\]
\end{lem}

\begin{proof}
	Set \(z_i:=u_i/v_i\). By symmetry,
	\[
	\mathbf u^T A\mathbf u
	-
	(A\mathbf v)\cdot\frac{\mathbf u^2}{\mathbf v}
	=
	-\frac12
	\sum_{i,j=1}^r
	a_{ij}v_iv_j(z_i-z_j)^2 .
	\]
	The diagonal terms vanish, and \(a_{ij}\le0\) for \(i\ne j\). Hence the
	right-hand side is nonnegative.
\end{proof}

The componentwise system version is obtained by summing
Lemma~\ref{lem:picone-discrete} over the components. Symmetry and the
nonpositive off-diagonal sign condition are essential in the discrete Picone
identity; an \(M\)-matrix property alone, without symmetry, is not sufficient.

\section{Appendix C. The linearized concave term}
\label{sec:appendix-concave}

We record the estimate which justifies the linearization of
\(u^{q-1}\), \(1<q<2\), on the regular positive cone. The Hopf-type lower
bound in the definition of \(\mathbb U^o\) compensates for the boundary
singularity of the derivative.

\begin{lem}[Boundedness of the linearized concave term]
	\label{lem:concave-differentiability}
	Let \(1<q<2\), and set
	\[
	g(u):=\bigl((u^1)^{q-1},\ldots,(u^m)^{q-1}\bigr)^T .
	\]
	For every \(u\in\mathbb U^o\), the formal derivative
	\[
	\bigl\langle Dg(u)\xi,\eta\bigr\rangle_m
	=
	(q-1)\sum_{k=1}^m
	\int_\Omega
	(u^k)^{q-2}\xi^k\eta^k\,dx
	\]
	defines a bounded operator \(Dg(u):\mathbb W\to\mathbb W^*\). In
	particular, \(D_u\mathcal F(u,\lambda)\) is well defined for every
	\(u\in\mathbb U^o\) and \(\lambda\in\mathbb R\).
\end{lem}

\begin{proof}
	Since \(u\in\mathbb U^o\), each component satisfies
	\(u^k\ge c_kd_\Omega\) in \(\Omega\). Hence, because \(1<q<2\),
	\[
	(u^k)^{q-2}\le C d_\Omega^{q-2}\le C d_\Omega^{-1}.
	\]
	By the Hardy inequality,
	\[
	\int_\Omega
	(u^k)^{q-2}|\xi^k||\eta^k|\,dx
	\le
	C
	\left\|\frac{\xi^k}{d_\Omega}\right\|_{L^2}
	\|\eta^k\|_{L^2}
	\le
	C\|\xi^k\|_{H_0^1}\|\eta^k\|_{H_0^1}.
	\]
	Summing over \(k=1,\ldots,m\) gives
	\[
	\bigl|
	\langle Dg(u)\xi,\eta\rangle_m
	\bigr|
	\le
	C\|\xi\|_{\mathbb W}\|\eta\|_{\mathbb W}.
	\]
	The assertion follows.
\end{proof}

\section{Appendix D. Linearized closedness along Galerkin limits}
\label{sec:appendix-CD}

We verify the closedness properties used in the passage from Galerkin
saddle-point pairs to the limiting elliptic system. Throughout this appendix,
functions defined on \(\Omega_r\) are extended by zero to \(\Omega\). We use
that the admissible domains \(\Omega_r\Subset\Omega\) satisfy a uniform Hardy
inequality and exhaust \(\Omega\) from inside.

\begin{lem}[Residual and linearized closedness]
	\label{lem:CD-elliptic}
	Let \(r_n\to\infty\), \(\lambda_n\to\lambda\),
	\(u_n\in\mathbb S_{r_n}^o\), \(v_n\in\mathbb S_{r_n}^o\), and assume that
	\[
	u_n\to u \quad\text{strongly in }\mathbb W,
	\qquad
	v_n\rightharpoonup v \quad\text{weakly in }\mathbb W,
	\qquad
	u,v\in\mathbb S\setminus\{0\}.
	\]
	Then, for every \(\zeta\in\mathbb W\) and every
	\(\zeta_n\in\mathbb W_{r_n}\) such that
	\(\zeta_n\to\zeta\) strongly in \(\mathbb W\),
	\[
	\langle\mathcal F(u_n,\lambda_n),\zeta_n\rangle
	\to
	\langle\mathcal F(u,\lambda),\zeta\rangle .
	\]
	
	Assume, in addition, that \(u\in\mathbb U^o\) and that \((u_n)\) satisfies
	the discrete barrier estimate
	\[
	u_n^k\ge c\,d_{\Omega_{r_n}}
	\qquad\text{in }\Omega_{r_n},\quad k=1,\ldots,m,
	\]
	with \(c>0\) independent of \(n\), where
	\(d_{\Omega_{r_n}}(x):=\operatorname{dist}(x,\partial\Omega_{r_n})\).
	Then, for every \(\xi\in\mathbb W\) and every
	\(\xi_n\in\mathbb W_{r_n}\) such that
	\(\xi_n\to\xi\) strongly in \(\mathbb W\),
	\[
	\bigl\langle
	D_u\mathcal F(u_n,\lambda_n)\xi_n,v_n
	\bigr\rangle
	\to
	\bigl\langle
	D_u\mathcal F(u,\lambda)\xi,v
	\bigr\rangle .
	\]
\end{lem}

\begin{proof}
	The residual convergence follows from the continuity of \(a_m\), the
	strong convergence \(u_n\to u\) in \(\mathbb W\), compact subcritical
	embeddings, and the growth assumptions on \(f\) and \(g\).
	
	We prove the convergence of the linearized identities. The terms
	\(a_m(\xi_n,v_n)\) and
	\(\langle f_u(u_n)\xi_n,v_n\rangle_m\) pass to the limit by the strong
	convergence of \(\xi_n\), the weak convergence of \(v_n\), compact
	subcritical embeddings, and the growth assumption on \(f_u\). It remains to
	treat
	\[
	\sum_{k=1}^m
	\int_{\Omega_{r_n}}
	(u_n^k)^{q-2}\xi_n^k v_n^k\,dx .
	\]
	Fix \(K\Subset\Omega\). For all large \(n\), \(K\subset\Omega_{r_n}\), and
	the interior exhaustion gives \(d_{\Omega_{r_n}}\ge\delta_K>0\) on \(K\).
	Hence the barrier estimate implies \(u_n^k\ge c\delta_K\) on \(K\). Since
	\(1<q<2\), the functions \((u_n^k)^{q-2}\) are uniformly bounded on \(K\)
	and converge a.e. to \((u^k)^{q-2}\). Therefore
	\[
	(u_n^k)^{q-2}\xi_n^k
	\to
	(u^k)^{q-2}\xi^k
	\quad\text{strongly in }L^2(K).
	\]
	As \(v_n^k\rightharpoonup v^k\) weakly in \(L^2(K)\), we get
	\[
	\int_K (u_n^k)^{q-2}\xi_n^k v_n^k\,dx
	\to
	\int_K (u^k)^{q-2}\xi^k v^k\,dx .
	\]
	
	It remains to control the boundary part. By the discrete barrier estimate,
	\(1<q<2\), and the uniform Hardy inequality,
	\[
	\int_E (u_n^k)^{q-2}|\xi_n^k||v_n^k|\,dx
	\le
	C\|\xi_n^k\|_{H_0^1(\Omega_{r_n})}\|v_n^k\|_{L^2(E)}
	\]
	for every measurable \(E\subset\Omega_{r_n}\). Taking
	\(E=\Omega_{r_n}\setminus K\), using the compactness of the embedding
	\(\mathbb W\hookrightarrow\mathbb L^2\), which gives
	\(v_n\to v\) strongly in \(\mathbb L^2\), and then choosing
	\(K\uparrow\Omega\), we make this contribution uniformly small. The same
	Hardy estimate, now using the Hopf-type lower bound for \(u\), controls the
	limiting integral on \(\Omega\setminus K\). Hence
	\[
	\int_{\Omega_{r_n}}
	(u_n^k)^{q-2}\xi_n^k v_n^k\,dx
	\to
	\int_\Omega
	(u^k)^{q-2}\xi^k v^k\,dx .
	\]
	Summing over \(k=1,\ldots,m\) and using \(\lambda_n\to\lambda\), we obtain
	the desired convergence.
\end{proof}

\begin{rem}
	In the extended \((PS)_e\)-sequences used in
	Lemma~\ref{lem:PS-application}, the required discrete barrier estimate is
	obtained by comparison with the positive scalar sublinear Galerkin
	solution. Hence Lemma~\ref{lem:CD-elliptic} verifies the linearized
	closedness condition \({\rm (CD)}\) for the elliptic Galerkin limits needed
	in the proof of Theorem~\ref{Thm42}.
\end{rem}

\end{document}